\documentclass[11pt]{article}

\usepackage{amsmath}
\usepackage{amssymb}
\usepackage{amsthm}
\usepackage{graphics}
\usepackage[latin1]{inputenc}
\usepackage[english]{babel}
\usepackage[T1]{fontenc}

\usepackage{enumerate}

\newtheorem{thm}{Theorem}[section]
\newtheorem{prop}{Proposition}[section]
\newtheorem{coro}{Corollary}[section]
\newtheorem{lemma}{Lemma}[section]
\newtheorem{rem}{Remark}[section]

\newtheorem{hypo}{Hypothesis}[section]

\newtheorem*{conjecture}{Conjecture}

\newcommand{\R}{\mathbb{R}}             
\newcommand{\N}{\mathbb{N}}             
\newcommand{\C}{\mathbb{C}}             

\newcommand{\m}{\mu_{m}}
\newcommand{\n}{\nu_{n}}

\newcommand{\half}{\frac{1}{2}}

\newcommand{\tr}{\frac{3}{4}}

\newcommand{\suchthat}{\, \mid \,}

\newcommand{\ds}{\displaystyle}

\newcommand{\Section}[1]{\section{#1} \setcounter{equation}{0}}

\begin{document}

\title{The anisotropic Calder\'{o}n problem for singular metrics of warped product type: the borderline between uniqueness and invisibility}
\author{Thierry Daud\'e \footnote{Research supported by the ANR Iproblems, No. ANR-13-JS01-0006 and the ANR JCJC Horizons, ANR-16-CE40-0012-01}, Niky Kamran \footnote{Research supported by NSERC grant RGPIN 105490-2011} and Francois Nicoleau \footnote{Research supported by the French GDR Dynqua}}





\maketitle


\begin{abstract}
In this paper, we investigate the anisotropic Calder\'{o}n problem on cylindrical Riemannian manifolds with boundary having two ends and equipped with singular metrics of (simple or double) warped product type, that is whose warping factors only depend on the horizontal direction of the cylinder. By singular, we mean that these factors are only assumed to be positive almost everywhere and to belong to some $L^p$ space with $0 < p \leq \infty$. Using recent developments on the Weyl-Titchmarsh theory for singular Sturm-Liouville operators, we prove that the local Dirichlet-to-Neumann maps at each end are well defined and determine the metric uniquely if: \\
\textbf{1.} (Doubly warped product case) the coefficients of the metric are $L^\infty$ and bounded from below by a positive constant. \\
\textbf{2.} (Warped product case) the coefficients of the metrics belong to a critical $L^p$ space where $p < \infty$ depends on the dimension of the compact fibers of the cylinder.

Finally, we show (in the warped product case and for zero frequency) that these uniqueness results are sharp by giving simple counterexamples for a class of singular metrics whose coefficients do not belong to the critical $L^p$ space. All these counterexamples lead to a region of space that is invisible to boundary measurements.
\end{abstract}


\vspace{1cm}

\noindent \textit{Keywords}. Anisotropic Calderon problem, singular Sturm-Liouville problems, Weyl-Titchmarsh function. \\


\noindent \textit{2010 Mathematics Subject Classification}. Primaries 81U40, 35P25; Secondary 58J50.

\tableofcontents


\Section{Introduction and model}

This paper is a continuation of our previous work \cite{DKN2} and is devoted to the study of anisotropic Calder\'{o}n  problems on cylindrical Riemannian manifolds $(M,g)$, for a class of metrics $g$ with singular coefficients.  We recall that the anisotropic Calder\'on inverse problem consists in determining the properties of a medium, for instance the electrical conductivity of a body, by making electrical or voltage  measurements at its boundary. As already noted in  \cite{LeU}, the Calder\'{o}n problem has a natural geometric reformulation in terms of the Dirichlet-to-Neumann map (DN map) on a Riemannian manifold. We also refer to the surveys \cite{GT2, KS2, Sa, U1} for the current state of the art on the anisotropic Calder\'on problem and  to \cite{DSFKSU, DSFKLS, GSB, GT1, KS1, LaTU, LaU, LeU} for important contributions. In order to state our results, let us recall some standard definitions:

\vspace{0.2cm}

Let $(M, g)$ be an $n$-dimensional smooth compact connected Riemannian manifold with smooth boundary $\partial M$, and let $-\Delta_{g}$ be the positive Laplace-Beltrami operator on $(M,g)$, given in local coordinates $(x^i)_{i = 1,\dots,n}$ by
$$
  -\Delta_g = -\frac{1}{\sqrt{|g|}} \ \  \partial_i \left( \sqrt{|g|} g^{ij} \partial_j \right),
$$
where  $|g| = \det \left(g_{ij}\right)$ is the determinant of the metric tensor $(g_{ij})$ and where $\left(g^{ij}\right)$ is the inverse of $(g_{ij})$. (We use the Einstein summation convention throughout this paper.)

\vspace{0.2cm}
The Laplace-Beltrami operator $-\Delta_g$ with Dirichlet boundary conditions on $\partial M$ is self-adjoint on $L^2(M, dVol_g)$ and has pure point spectrum $\{ \lambda_j\}_{j \geq 1}$ with $0 < \lambda_1 < \lambda_2 \leq \dots \leq \lambda_j \to +\infty$, (see for instance \cite{KKL}). For $\lambda \not= \lambda_j$, we consider the Dirichlet problem
\begin{equation} \label{Eq000}
  \left\{ \begin{array}{cc} -\Delta_g u = \lambda u, & \textrm{on} \ M, \\ u = \psi, & \textrm{on} \ \partial M. \end{array} \right.
\end{equation}

 \vspace{0.2cm}
 It is well known \cite{Sa} that, for any $\psi \in H^{1/2}(\partial M)$, there exists a unique weak solution $u \in H^1(M)$ of the Dirichlet problem (\ref{Eq000}), \textit{i.e.} $u$ satisfies for all $v \in C_0^\infty(M)$,
\begin{equation} \label{Weak-DirichletPb}
  \int_M [ g^{ij} \partial_i u \partial_j v - \lambda u v ]\,dVol_g = 0,
\end{equation}
where $dVol_g = \sqrt{|g|} dx$ is the Riemannian volume element. So, we can define the Dirichlet-to-Neumann  map as the operator $\Lambda_{g}(\lambda)$ from $H^{1/2}(\partial M)$ to $H^{-1/2}(\partial M)$ given by
\begin{equation} \label{DN-Abstract}
  \Lambda_{g}(\lambda) (\psi) = \left( \partial_\nu u \right)_{|\partial M}.
\end{equation}
Here, $u$ is the unique solution of (\ref{Eq000}) and $\left( \partial_\nu u \right)_{|\partial M}$ is its normal derivative with respect to the unit outer normal $\nu$ on $\partial M$. Note that this normal derivative  has to be understood  in the weak sense as an element of $H^{-1/2}(\partial M)$ via the bilinear form
$$
  \left\langle \Lambda_{g}(\lambda) \psi | \phi \right \rangle = \int_M  \left( \langle du, dv \rangle_g - \lambda uv \right)\  dVol_g,
$$
where $\psi, \phi  \in H^{1/2}(\partial M)$,  $u$ is the unique solution of the Dirichlet problem (\ref{Eq000}), and where $v$ is any element of $H^1(M)$ such that $v_{|\partial M} = \phi$. Of course, when  $\psi$ is sufficiently smooth, this definition coincides with the usual one in local coordinates, that is
\begin{equation} \label{DN-Coord}
\partial_\nu u = \nu^i \partial_i u.
\end{equation}

We can also refine the definition of the DN map and introduce the \emph{partial} DN map on open sets  $\Gamma_D$ and $\Gamma_N$ of the boundary $\partial M$: the Dirichlet data are prescribed on $\Gamma_D$ and the Neumann data are measured on $\Gamma_N$. In other words, we consider the Dirichlet problem
\begin{equation} \label{Eq0}
  \left\{ \begin{array}{cc} -\Delta_g u = \lambda u, & \textrm{on} \ M, \\ u = \psi, & \textrm{on} \ \partial M, \end{array} \right.
\end{equation}
where $\psi \in H^{1/2}(\partial M)$ with $\textrm{supp}\,\psi \subset \Gamma_D$. The partial DN map is then defined as:
\begin{equation} \label{Partial-DNmap}
  \Lambda_{g,\Gamma_D,\Gamma_N}(\lambda) (\psi) = \left( \partial_\nu u \right)_{|\Gamma_N}.
\end{equation}

\vspace{0.5cm}
The anisotropic partial Calder\'on problem can now be stated as follows : \emph{does the knowledge of the partial DN map $\Lambda_{g,\Gamma_D, \Gamma_N}(\lambda)$ at a fixed frequency $\lambda$ determine the metric $g$?}

\vspace{0.2cm}
It is well known that as a consequence of a number of gauge invariances, the answer to this question is negative. Indeed, first of all, the partial DN map  $\Lambda_{g, \Gamma_D, \Gamma_N}(\lambda)$ is invariant when the metric $g$ is pulled back by any diffeomorphism of $M$ which is equal  to the identity on $\Gamma_D \cup \Gamma_N$, \textit{i.e.}
\begin{equation} \label{Inv-Diff}
  \forall \phi \in \textrm{Diff}(M) \ \textrm{such that} \ \phi_{|\Gamma_D \cup \Gamma_N} = Id, \quad \Lambda_{\phi^*g, \Gamma_D, \Gamma_N}(\lambda) = \Lambda_{g, \Gamma_D, \Gamma_N}(\lambda).
\end{equation}

Secondly, in dimension 2 and for $\lambda=0$, the conformal invariance of the Laplace-Beltrami operator leads to another gauge invariance of the partial DN map: if $dim \ M =2$, for any smooth function $c>0$, one has:
$$
\Delta_{cg} = \frac{1}{c} \Delta_g.
$$
Then, it follows easily that:
\begin{equation} \label{Inv-Conf}
\forall c \in C^\infty(M) \ \textrm{such that} \ c >0 \ \textrm{and} \ c_{|\Gamma_N} = 1, \quad \Lambda_{c g, \Gamma_D, \Gamma_N}(0) =
\Lambda_{g, \Gamma_D, \Gamma_N}(0).
\end{equation}

\vspace{0.3cm}
As a consequence, due to these gauge invariances, the \emph{anisotropic Calder\'on conjecture} is reformulated as follows :  \\

\begin{conjecture}
Let $M$ be a smooth compact connected manifold with smooth boundary $\partial M$ and let $g,\, \tilde{g}$ denote smooth Riemannian metrics on $M$. Let $\Gamma_D, \Gamma_N$ be open subsets of $\partial M$. Assume that $\lambda \in \R$ does not belong to $\sigma(-\Delta_g) \cup \sigma(-\Delta_{\tilde{g}})$ for the Dirichlet realizations of $\sigma(-\Delta_g)$ and $\sigma(-\Delta_{\tilde{g}})$, and suppose that
$$
  \Lambda_{g,\Gamma_D, \Gamma_N}(\lambda) = \Lambda_{\tilde{g},\Gamma_D, \Gamma_N}(\lambda).
$$
Does it follow that
$$
  g = \tilde{g},
$$
up to the gauge invariance (\ref{Inv-Diff}) if $\dim M \geq 3$ and up to the gauge invariances (\ref{Inv-Diff}) - (\ref{Inv-Conf}) if $\dim M = 2$ and $\lambda = 0$? \end{conjecture}

\vspace{0.5cm}
Before stating our results, let us give a brief survey of some of the most important known contributions to the Calder\'{o}n conjecture. The main results for this conjecture apply when the frequency $\lambda=0$ and for full data $(\Gamma_D=\Gamma_N= \partial M)$, or for local data $(\Gamma_D=\Gamma_N = \Gamma$, where $\Gamma$ is any open set of $\partial M$). In dimension $2$, for compact and connected surfaces and for  $\lambda=0$, the anisotropic Calder\'on conjecture has been proved for full or local data, (see \cite{LaU, LeU}). In dimension greater than or equal to $3$, for real-analytic Riemannian manifolds  or for compact connected Einstein manifolds with boundary, it has likewise been shown that the local DN map determines uniquely the metric up to the natural gauge invariances, (see \cite{LeU, LaU, LaTU, GSB}.

\vspace{0.2cm}
In the general case of smooth metrics without any analyticity's assumptions, the anisotropic Calder\'on conjecture is still a major open problem, either for full or local data. Nevertheless, for conformally transversally anisotropic manifolds and for metrics belonging to the same conformal class, some important results have been obtained recently \cite{DSFKSU, DSFKLS, DSFKLLS, KS1}.

\vspace{0.2cm}
In the case  of partial data measured on \emph{disjoint sets}, the anisotropic conjecture has been answered negatively in \cite{DKN2, DKN3, DKN4}. More precisely, given a smooth compact connected Riemannian manifold with boundary $(M,g)$, of dimension $n\geq 3$, there exist in the conformal class of $g$ an infinite number of Riemannian metrics $\tilde{g}= c^4 g$ having the same partial DN map $\Lambda_{g,\Gamma_D, \Gamma_N}(\lambda)$, when  $\Gamma_D   \cap  \Gamma_N = \emptyset$ and satisfying the condition $\overline{\Gamma_D \cup \Gamma_N} \ne \partial M$. The conformal factors $c^4$ that lead to these non-uniqueness results satisfy a nonlinear elliptic PDE of Yamabe type on $(M,g)$:
\begin{equation}\label{Yamabe}
  \left\{ \begin{array}{cc} \Delta_{g} c^{n-2} + \lambda ( c^{n-2} - c^{n+2}) = 0, & \textrm{on} \ M, \\
	c = 1, & \textrm{on} \ \Gamma_D \cup \Gamma_N. \end{array} \right.
\end{equation}

In other words, when $\Gamma_D$ and $\Gamma_N$ are disjoint, one can exhibit a natural gauge invariance governed by the equation (\ref{Yamabe}). Nevertheless, for a class of cylindrical Riemannian manifolds with boundary having two ends, and equipped with a suitable warped product metric, one can also construct counterexamples to uniqueness modulo this natural gauge invariance, \cite{DKN3,DKN4}.

\vspace{0.2cm}
To conclude, let us mention several papers dealing with the Calder\'on problem for \emph{singular} metrics or conductivities, the topic which is the object of this paper. In dimension $2$, the original Calder\'on problem for conductivities was solved by Astala and P\"aiv\"arinta in \cite{AP}. The authors showed that a measurable isotropic conductivity bounded uniformly from below and above is uniquely determined by the global DN map. These 2D results were later generalized in \cite{ALP2} leading to a precise borderline between uniqueness and invisibility results in the Calder\'on problem. In dimensions higher than $3$, Habermas and Tataru \cite{HaTa} showed uniqueness in the global Calder\'on problem for uniformly elliptic isotropic conductivities that are Lipschitz and close to the identity. The latter condition was relaxed by Caro and Rogers in \cite{CaRo}. In dimensions $3$ and $4$, these results were slightly improved by Habermas in \cite{Ha} to the case of conductivities that belong to $W^{1,n}$. In \cite{KrUh2}, Krupchyk and Uhlmann proved that the magnetic and electric $L^\infty$ potentials of a magnetic Schr\"odinger operator on a conformally transversally manifold are uniquely determined by the global DN map. Even more recently, Santacesaria \cite{San} announced a new strategy (an higher dimensional analog to the one used in \cite{AP}) to adress the question whether a uniformly elliptic isotropic $L^\infty$ conductivity could be uniquely determined by the global DN map in dimensions higher than $3$. Related to the partial Calder\'on problem, Krupchyk and Uhlmann in \cite{KrUh1} proved that an isotropic conductivity with - roughly speaking - $\frac{3}{2}$ derivatives in the $L^2$ sense is uniquely determined by a DN map measured on possibly very small subset of the boundary.


\vspace{0.5cm}
In this paper, we are interested in studying the anisotropic Calder\'on problem for (doubly) warped product metrics of the type already encountered in \cite{DKN3, DKN4} but with less regularity. Precisely, we consider a cylindrical Riemannian manifold $(M,g)$ of the form
\begin{equation} \label{Cylinder0}
  M = [0,1] \times K_1 \times K_2,
\end{equation}	
where $(K_j, g_j)$ are closed $n_j$-dimensional smooth Riemannian manifolds. We consider only manifolds of dimension higher than $3$ by assuming that $n_1 \geq 1, \ n_2 \geq 0, \ n_1 + n_2 \geq 2$. We suppose that $M$ is equipped with Riemannian metrics of the following form:
\begin{equation} \label{Metric0}
  g = h_1(x) dx^2 + h_1(x) g_1 + h_2(x) g_2,
\end{equation}
where $h_1, \ h_2$ are \emph{measurable positive functions depending only on the first variable $x \in [0,1]$}. Note that for $n_2 \geq 1$ and $h_1 \ne h_2$, this class of doubly warped product metrics does not enter the framework of conformally transversaly anisotropic metrics studied in \cite{DSFKSU, DSFKLS, DSFKLLS}. Conversely, note that for $n_2 = 0$ or $n_2 \geq 1$ and $h_1 = h_2$, we recover the usual warped product metrics. Another important point to mention is that the boundary $\partial M$ of the manifold $M$ is not connected and consists in the disjoint union of two copies of $K_1 \times K_2$ that we will call {\it{ends}}:
\begin{equation}\label{bord}
  \partial M = \Gamma_0 \cup \Gamma_1, \quad \Gamma_0 = \{0\} \times K_1 \times K_2, \quad \Gamma_1 = \{1\} \times K_1 \times K_2.
\end{equation}

In Section \ref{3}, we first work with a metric that is of uniformly elliptic signature and whose coefficients belong to $L^\infty$, \textit{i.e.} we assume that there exist two positive constants $0 < c < C $ such that $c \leq h_j(x) \leq C$ almost everywhere in $[0,1]$. Clearly, the Laplace-Beltrami operator $-\Delta_g$ is then uniformly elliptic on $M$ and it is well known that the Dirichlet problem (\ref{Eq000}) has a unique solution $u \in H^1(M)$ for any boundary data $\psi \in H^\half(\partial M)$ (see \cite{GT}). If $\lambda$ is a frequency not belonging to the discrete Dirichlet spectrum of $-\Delta_g$, we  define as previously the local DN maps at each end $\Lambda_{g, \Gamma_0, \Gamma_0}(\lambda)$ and $\Lambda_{g, \Gamma_1, \Gamma_1}(\lambda)$. In other words, we study the Calder\'on problem when the Dirichlet data are prescribed on the end $\Gamma_0$ or $\Gamma_1$ and the Neumann data are measured on the same end.

Our main result is the following:

\begin{thm} \label{MainThm}
Let $(M,g)$ and $(M,\tilde{g})$ denote two Riemannian manifolds of the form (\ref{Cylinder0})-(\ref{Metric0}). Assume that:
$$
 \Lambda_{g, \Gamma_j,\Gamma_j}(\lambda) = \Lambda_{\tilde{g}, \Gamma_j,\Gamma_j}(\lambda) \ \ ,\ \ {\rm{for}} \ j=0 \ {\rm{or}} \ j= 1.
$$
Then, $$
g= \tilde{g}.
$$
\end{thm}

\vspace{0.2cm}
\noindent We emphasize that the gauge invariance (\ref{Inv-Diff}) does not appear in the statement of the theorem since both metrics $g$ and $\tilde{g}$ have the particular form (\ref{Metric0}). Note that this class of metrics could be extended to metrics $g$ on $M$ having the form
\begin{equation} \label{MetricAlt0}
  g=h_0(x) dx^2 + h_1(x) g_1 + h_2(x) g_2,
\end{equation}
since any metric (\ref{MetricAlt0}) can always be written as (\ref{Metric0}) thanks to the change of variables $y = \int_0^x \sqrt{\frac{h_0(t)}{h_1(t)}} \ dt$ provided this change of variables is well defined. We give the corresponding uniqueness result in Theorem \ref{MainThm-Extended1}. Note also that Theorem \ref{MainThm} is an extension to our models of the well-known $2$ dimensional results obtained in \cite{AP} for measurable conductivities that are uniformly bounded from below and above.

\vspace{0.5cm}
Let us briefly outline the strategy of the proof of Thm \ref{MainThm}. On one hand, thanks to the cylindrical symmetry of $(M,g)$ and its doubly warped product structure, we may separate the radial variable, \textit{i.e.} look for the solutions $u$ of the Helmholtz equation $-\Delta_g u = \lambda u$ of the form:
$$
u = \sum_{m,n=0}^{+\infty} u_{mn}(x) \ \Phi_m \Psi_n,
$$
where
$(\Phi_m)$, (resp. $(\Psi_n)$) is a Hilbert basis of harmonics of the Laplace-Beltrami operator $-\Delta_{g_1}$, (resp. $-\Delta_{g_2}$), with associated eigenvalues $\m$, (resp. $\n)$. 

Up to some suitable boundary conditions, the functions $u_{mn}(x)$ satisfy on $[0,1]$ the singular Sturm-Liouville equation with respect to $x$:
\begin{equation} \label{SLintro}
   - \frac{1}{\sqrt{h}} \left( \sqrt{h}u_{mn}'\right)'  + (  \n \frac{h_1}{h_2}  - \lambda h_1 ) u_{mn} = -\m u_{mn},
\end{equation}
where we have set $h = h_1^{n_1 -1} h_2^{n_2}$, and where $-\m$ plays the role of a spectral parameter. It turns out that these equations fit into the framework of the recent work by Eckhardt-Gesztesy-Nichols-Teschl \cite{EGNT1,EGNT2} on inverse spectral theory for Sturm-Liouville operators with distributional potentials.

On the other hand, the local DN maps $\Lambda_{g, \Gamma_0,\Gamma_0}(\lambda)$ and $\Lambda_{g, \Gamma_1,\Gamma_1}(\lambda)$ can be also diagonalized onto the same Hilbert basis of harmonics $Y_{mn} = \Phi_m \Psi_n$ and turn out to be operators of multiplication by the singular Weyl-Titchmarsh functions associated to the Sturm-Liouville equations (\ref{SLintro}) with Dirichlet boundary conditions. Once this is established, Theorem \ref {MainThm} is obtained as a natural consequence of the uniqueness results for spectral measures obtained in \cite{EGNT2}, Theorem 3.4 and a slight extension of the complex angular momentum (CAM) method already used in \cite{DKN1, DN3, DN4, DGN}.

In Section \ref{4}, we consider the particular case of cylindrical manifolds
\begin{equation} \label{Cylinder01}
  M = [0,1] \times K,
\end{equation}	
equipped with Riemannian warped product metrics :
\begin{equation} \label{Metric01}
  g = h_1(x) [dx^2 + g_K],
\end{equation}
where $(K, g_K)$ is a closed $n$-dimensional smooth Riemannian manifold and $h_1$ is a measurable function depending only on the first variable $x \in [0,1]$ that satisfies
\begin{equation} \label{MainCond}
  h_1(x) > 0 \ \textrm{a.e.}, \quad h_1^{\frac{n-1}{2}} \in L^1(0,1), \quad \frac{1}{h_1^{\frac{n-1}{2}}} \in L^1(0,1).
\end{equation}
Note that under these above assumptions, the Laplace-Beltrami operator $-\triangle_g$ is not uniformly elliptic in general. We study the anisotropic Calder\'on problem at zero frequency for this class of singular warped product Riemannian manifolds. We first prove that given Dirichlet data $\psi \in H^2(\partial M)$, there exists a unique solution $u \in \dot{H}^1(M)$ (the homogeneous Sobolev space of order $1$) of the Dirichlet problem
$$
  \left\{ \begin{array}{cc} -\Delta_g u = 0, & \textrm{on} \ M, \\ u = \psi, & \textrm{on} \ \partial M. \end{array} \right.
$$
In fact, under the additional assumption
\begin{equation} \label{MainCondAdd}
  h_1^{\frac{n+1}{2}} \in L^1(0,1),
\end{equation}
we can prove that for Dirichlet data $\psi \in H^2(\partial M)$, there exists a unique solution $u \in H^1(M)$ (the usual Sobolev space of order $1$) of the Dirichlet problem. Using these results we can then define the DN map as a bounded linear operator from $H^2(M)$ into $L^2(M)$. 

Our second main uniqueness theorem is:

\begin{thm} \label{MainThm2}
Let $(M,g)$ and $(M,\tilde{g})$ denote two Riemannian manifolds of the form (\ref{Cylinder01})-(\ref{Metric01}). Assume that (\ref{MainCond}) holds and that
$$
 \Lambda_{g, \Gamma_j,\Gamma_j} = \Lambda_{\tilde{g}, \Gamma_j,\Gamma_j} \ \ ,\ \ {\rm{for}} \ j=0 \ {\rm{or}} \ j= 1.
$$
Then, $$
g= \tilde{g}.
$$
\end{thm}

We emphasize that the proofs of these results once again crucially rely on the separation of the radial variable, the CAM method and the uniqueness results of \cite{EGNT2}. Finally, we conclude Section \ref{4} by giving some counterexamples to uniqueness for the class of metrics (\ref{Cylinder01})-(\ref{Metric01}) whenever one of the assumptions (\ref{MainCond}) is not satisfied. These counterexamples are examples of invisibility phenomena, that is existence of a bounded region that cannot be seen by boundary measurements. 

\vspace{0.2cm}
Our paper is organized as follows. In Section 2, we recall some inverse spectral results for singular Sturm-Liouville problems obtained in \cite{EGNT2} as well as some classical results on entire functions with certain growth orders at infinity that will be useful in the later sections. The latter results can be found for instance in \cite{Lev}. In Section 3, we give the details of the separation of variables procedure in the case of doubly warped product metrics that leads to a convenient expression of the global DN map as a Hilbert sum of operators of multiplication by some Weyl-Titchmarsh functions associated to singular Sturm-Liouville equations. We then provide the proof of Theorem \ref{MainThm} by performing a CAM method and applying the results of \cite{EGNT2}. In Section 4, we consider the case of warped product metrics and first show the existence and uniqueness of a solution in $\dot{H}^1(M)$ of the Dirichlet problem provided that the Dirichlet data are in $H^2(\partial M)$. We then prove Theorem \ref{MainThm2} following the same strategy as in Section \ref{3}. Finally, we exhibit some simple counterexamples to uniqueness for warped product metrics whose coefficients do not satisfy the minimal assumptions of Theorem \ref{MainThm2}.


\Section {Inverse spectral theory for Sturm-Liouville operators with measurable coefficients} \label{2}

In this section, we recall some basic facts on Sturm-Liouville operators with measurable coefficients. For an exposition of this theory (in the more general case of distributional potentials), we refer to (\cite{EGNT1, EGNT2, ET, Ze}). Thus we consider a differential operator on the interval $[0,1]$ given by
\begin{equation}\label{SL1}
\tau u = \frac{1}{r}  \left( -(pu')' +qu\right).
\end{equation}
We make the following assumptions:

\begin{hypo} \label{H}
  We assume that the functions $p,q,r$ are real-valued Lebesgue measurable on $[0,1]$ with $\frac{1}{p}, q, r \in L^1([0,1], dx)$ with $p \not=0$ and $r>0$ almost everywhere on  $[0,1]$.
\end{hypo}

\vspace{0.2cm}\par
By a solution of the equation $(\tau-z)u=0$ where $z \in \C$, we mean a function $u : [0,1] \rightarrow \C$ such that $u$ and $u^{[1]}:=pu'$ are absolutely continous (AC)  on $[0,1]$ and the equation is satisfied a.e on $[0,1]$. Given a solution $u$, we refer to $u^{[1]}$ as its quasi-derivative to distinguish it from the classical derivative $u'$ which is only defined a.e on $[0,1]$.

\vspace{0.2cm} \par
It is well known (see for instance \cite{Ze}, Theorems 2.21 and 2.31), that for each $c \in [0,1]$ and $\alpha,  \beta, z \in \C$, the equation $(\tau-z)u=0$ has a unique solution  with initial conditions $u(c) = \alpha$, $u^{[1]}(c) = \beta$. In particular, the operator $\tau$ is in the so-called limit circle case at the endpoints $x=0$ and $x=1$,  \textit{i.e.} all solutions of $(\tau-z)u=0$ lie in $L^2([0,1], r(x) dx)$.

For all $z \in \C$,  we can define two fundamental systems of solutions (FSS)
$$
  \{ c_0(x,z), s_0(x,z)\}, \quad \{ c_1(x,z), s_1(x,z)\},
$$
of $(\tau-z)u=0$ by imposing the Cauchy conditions
\begin{equation} \label{FSS}
  \left\{ \begin{array}{cccc} c_0(0,z) = 1, & c_0^{[1]}(0,z) = 0, & s_0(0,z) = 0, & s_0^{[1]}(0,z) = 1, \\
 	                  c_1(1,z) = 1, & c_1^{[1]}(1,z) = 0, & s_1(1,z) = 0, & s_1^{[1]}(1,z) = 1. \end{array} \right.
\end{equation}
It follows from (\ref{FSS}) that
\begin{equation} \label{Wronskian-FSS}
  W(c_0, s_0) = 1, \quad W(c_1, s_1) = 1, \quad \forall z \in \C,
\end{equation}
where $W(u,v)$ is the  modified Wronskian determinant  given by
\begin{equation}\label{wronskien}
W(u,v)(x) = u(x) v^{[1]} (x) - u^{[1]}(x) v(x).
\end{equation}

We also recall some simple estimates (see Theorem 2.5.3 in \cite{Ze}).

\begin{lemma} \label{2-EstFSS}
The functions $z \mapsto c_j(x,z), \, c_j^{[1]}(x,z), \, s_j(x,z), \, s_j^{[1]}(x,z)$ are entire functions of order $\frac{1}{2}$. More precisely, we have for all $|z| \geq 1$
$$
  |c_j(x,z)|, \, |s_j^{[1]}(x,z)| \leq C e^{A \sqrt{|z|}}, \quad c_j^{[1]}(x,z) \leq C \sqrt{|z|} e^{A \sqrt{|z|}}, \quad s_j(x,z) \leq \frac{C}{\sqrt{|z|}} e^{A \sqrt{|z|}},
$$
where $A = \frac{1}{2} \int_0^1 \left( \frac{1}{p(s)} + r(s) \right) ds$ and $C$ denotes constants independent of $z$.
\end{lemma}

As a consequence, every  solution $u$ of $(\tau-z)u=0$, as well as its quasi-derivative $u^{[1]}$, are entire functions with respect to the variable $z \in \C$, of order at most $1/2$, i.e there exists $A,C$ such that, for all $x \in [0,1]$ and $z \in \C$,
\begin{equation} \label{ordreundemi}
|u(x)| \leq C e^{A \sqrt{|z|}} \ ,\ |u^{[1]}(x)| \leq C e^{A \sqrt{|z|}}.
\end{equation}

\vspace{0.2cm}
We define then  the characteristic function of the equation $(\tau -z)u=0$  with Dirichlet boundary conditions $u(0)=0, \ u(1)=0$  by
\begin{equation} \label{Char}
  \Delta(z) = W(s_0, s_1).
\end{equation}
The characteristic function $z \mapsto \Delta(z)$ is entire of order $1/2$ on the complex plane $\C$, and its zeros $(\alpha_{k})_{k \geq 1}$ correspond to  the eigenvalues of the self-adjoint operator $S$ on $L^2([0,1]; r(x) dx)$, given by
$$
Su=\frac{1}{r}  \left( -(pu')' +qu\right),
$$
with the Dirichlet boundary conditions  $u(0)=0$ and $u(1)=0$.

\vspace{0.2cm}

For $z \not=\alpha_k$, we next define two Weyl-Titchmarsh functions $M(z)$ and $N(z)$ by the following classical prescriptions. Let the Weyl solutions $\Psi$ and $\Phi$ be the unique solutions of $(\tau-z)u=0$ having the form
\begin{equation} \label{WeylFunction}
  \Psi(x,z) = c_0(x,z) + M(z) s_0(x,z) , \quad  \Phi(x,z) = c_1(x,z) - N(z) s_1(x,z) ,
\end{equation}
which satisfy the Dirichlet boundary condition at $x = 1$ and $x=0$ respectively. Then a short calculation using (\ref{FSS}) shows that the Weyl-Titchmarsh functions $M(z)$ and $N(z)$ are uniquely defined by:
\begin{equation} \label{WT}
  M(z) = - \frac{W(c_0, s_1)}{\Delta(z)}=  - \frac{c_0(1,z)}{s_0(1,z)}, \quad N(z) = -\frac{W(c_1, s_0)}{\Delta(z)}= \frac{c_1(0,z)}{s_1(0,z)}.
\end{equation}

Let us also introduce the functions $D(z) = W(c_0, s_1)$ and $E(z) = W(c_1, s_0)$ which also turn out to be entire functions of the variable $z$ of order $1/2$. Let us denote by $(\beta_{k})_{k \geq 1}$ and $(\gamma_{k})_{k \geq 1}$ the zeros of $D(z)$ and $E(z)$ respectively. They correspond to the eigenvalues of the self-adjoint operator $S$ with mixed Dirichlet and Neumann boundary conditions at $x=0$ and $x=1$. Precisely, $(\beta_k)_{k \geq 1}$ are the eigenvalues of $S$ associated to the Neumann boundary conditions at $0$ and Dirichlet noundary conditions at $1$, and $(\gamma_k)_{k \geq 1}$ are the eigenvalues of $S$ associated to the Dirichlet boundary conditions at $0$ and Neumann boundary conditions at $1$. We thus have the following expressions for the Weyl-Titchmarsh functions
\begin{equation}\label{MD}
  M(z) = - \frac{D(z)}{\Delta(z)} , \quad N(z) = - \frac{E(z)}{\Delta(z)}.
\end{equation}

We gather in the following Lemmas some useful results on $\Delta(z), \ D(z), \ E(z)$ that will important later.

\begin{lemma} \label{2-Hadamard}
The functions $z \mapsto \Delta(z), \, D(z),  \, E(z)$ are entire functions of order $\frac{1}{2}$ that can be written as
$$
  \Delta(z) = C_1 z^{m_1} \prod_{k=1}^\infty \left( 1 - \frac{z}{\alpha_k} \right), \quad D(z) = C_2 z^{m_2} \prod_{k=1}^\infty \left( 1 -  \frac{z}{\beta_k} \right), \quad E(z) = C_3 z^{m_3} \prod_{k=1}^\infty \left( 1 - \frac{z}{\gamma_k} \right),
$$
where $C_j$ are constants, $m_j = 0$ or $1$ and $(\alpha_k)_{k \geq 1}, \ (\beta_k)_{k \geq 1}, \ (\gamma_k)_{k \geq 1}$ are the simple eigenvalues (indexed in increasing order) of the self-adjoint Sturm-Liouville operator $S = \frac{1}{r} \left( -\frac{d}{dx} \left( p \frac{d}{dx} \right) + q \right)$ with Dirichlet or mixed Dirichlet-Neumann separated boundary conditions (see above) respectively.
\end{lemma}

\begin{proof}
  The fact that $z \mapsto \Delta(z), \, D(z), \ \, E(z)$ are entire functions of order $\frac{1}{2}$ is a consequence of Lemma \ref{2-EstFSS}. The second point is then a simple consequence of Hadamard's factorization Theorem (see \cite{Lev}, Lecture 4, Thm 1) and of the well-known facts on the self-adjoint operator $S$ with separated boundary conditions summarized in \cite{Ze}, Thm 4.3.1.
\end{proof}

\begin{lemma} \label{2-Weyl}
The distinct eigenvalues $(\alpha_k)_{k \geq 1}$, $(\beta_k)_{k \geq 1}$ and $(\gamma_k)_{k \geq 1}$ of $S$ satisfy the same Weyl law
$$
  \frac{\alpha_k}{k^2}, \, \frac{\beta_k}{k^2}, \, \frac{\gamma_k}{k^2} \to \pi^2, \quad k \to \infty. 	
$$
\end{lemma}

\begin{proof}
This is a well known fact that the Weyl law for the eigenvalues of $S$ does not depend on the choice of self-adjoint boundary conditions (see \cite{Ze}, Thm 4.3.1, (7)).
\end{proof}

\begin{prop} \label{2-Levin}
Let $(\lambda_k)_{k \geq 1}$ be a sequence of positive numbers satisfying $\ds\lim_{k \to \infty} \frac{k}{\lambda_k} = B$. If ${\displaystyle{\Pi(z) = \prod_{k = 1}^\infty \left( 1 - \frac{z^2}{\lambda_k^2} \right)}}$, then the indicator function of $\Pi$ defined by
$$
  h_{\Pi}(\theta) := \limsup_{r \to +\infty} \frac{\log |\Pi(re^{i\theta})|}{r},
$$
satisfies
$$
  h_{\Pi}(\theta) = \pi B |\sin(\theta)|.
$$
\end{prop}

\begin{proof}
  We refer to \cite{Lev}, Lecture 12, Theorem 2, for a proof of this result.
\end{proof}

Thanks to Lemmas \ref{2-Hadamard} and \ref{2-Weyl}, we can apply Proposition \ref{2-Levin} to the functions
\begin{equation} \label{2-Nota}
  \delta(z) = \Delta(z^2), \quad d(z) = D(z^2), \quad e(z) = E(z^2),
\end{equation}
in a straightforward way and we thus obtain

\begin{coro} \label{2-EstRef}
The indicator functions of $\delta(z), \, d(z), \, e(z)$ satisfy
$$
  h_{\delta}(\theta) = |\sin(\theta)|, \quad h_{d}(\theta) = |\sin(\theta)|, \quad h_{e}(\theta) = |\sin(\theta)|.
$$
\end{coro}
We stress the fact that these properties will be fundamental in the CAM method used below.

Let us come back now to some properties of the Weyl-Titchmarsh functions $M(z)$ and $N(z)$ that will be the main objects of study of this paper. We refer to \cite{EGNT1, EGNT2} for the details of the results we present now. First, we can associate with the Weyl-Titchmarsh function $M(z)$ (for instance) a unique Borel measure $\rho$ on $\R$ given by
\begin{equation}\label{mesure}
\rho (]a,b]) = \lim_{\delta \downarrow 0} \ \lim_{\epsilon \downarrow 0} \  \frac{1}{\pi} \int_{a+\delta}^{b+\delta} Im (M(\lambda + i \epsilon)) \ d\lambda.
\end{equation}
The operator $\mathcal{F}$ from $L^2([0,1], r(x) dx)$ onto $L^2(\R, d\rho)$ defined by
\begin{equation} \label{Funitary}
\mathcal{F} f(z) = \int_0^1 s_0(x,z) f(x) r(x) dx,
\end{equation}
is  unitary and the self-adjoint operator $S$ satisfies $S = \mathcal{F}^* M_{Id} \mathcal{F}$, where $M_{Id}$ denotes the operator of multiplication by the variable in $L^2(\R, d\rho)$. This measure $\rho$ is called {\it{the spectral measure}} of the operator $S$.

In order to estimate the WT function $M(z)$ at later stage, we shall use general facts about Herglotz functions, \textit{i.e.} analytic functions on the upper half complex plane $\C^+$ that satisfy: $\Im(z) > 0 \ \Longrightarrow \ \Im(M(z)) > 0$, associated to a self-adjoint Sturm-Liouville operator. We refer again to \cite{EGNT1, EGNT2} for the precise results corresponding to our framework. We recall that the spectral measure $\rho$ associated to the self-adjoint operator $H$ with Dirichlet boundary conditions is a Borel measure over $\R$ that satisfies
\begin{equation} \label{2-SpectralMeasure}
  \int_\R \frac{d\rho(\omega)}{1+\omega^2} < \infty,
\end{equation}
and that is connected to the WT function $M$ by the formula:
\begin{equation}\label{2-SpectralMeasure-M}
  M(z) = c + dz + \int_\R \left[ \frac{1}{\omega - z} - \frac{\omega}{1+\omega^2} \right] d\rho(\omega), \quad \forall z \in \C^+,
\end{equation}
where $c = \Re(M(i))$ and $d = \lim_{\eta \to \infty} \frac{M(i\eta)}{i\eta} \geq 0$.

For the self-adjoint operator $S$ with Dirichlet boundary conditions, the endpoints $0$ and $1$ are in the Limit Circle case. This implies that $d = 0$ in the formula (\ref{2-SpectralMeasure-M}) according to Corollary 9.8 in \cite{EGNT1}. Moreover, we know that the spectrum of $S$ is purely discrete, made of simple eigenvalues $(\alpha_k)_{k \geq 1}$ satisfying
$$
  - \infty < \alpha_1 < \alpha_2 < \dots < \alpha_k \to +\infty.
$$
Therefore, for all $\sigma < \alpha_1$, we have
\begin{equation}\label{2-SpectralMeasure-M1}
  M(z) = c + \int_\sigma^\infty \left[ \frac{1}{\omega - z} - \frac{\omega}{1+\omega^2} \right] d\rho(\omega), \quad \forall z \in \C^+.
\end{equation}
Let $z \in \C \setminus [\sigma,+\infty)$ and $\omega \in [\sigma,+\infty)$. Using (\ref{2-SpectralMeasure}) and the estimate
$$
  \left| \frac{1}{\omega - z} - \frac{\omega}{1+\omega^2} \right| \leq C_z \frac{1}{1+\omega^2},
$$
where $C_z$ is a constant depending on $z$, we can extend analytically (\ref{2-SpectralMeasure-M1}) for all $z \in \C \setminus [\sigma,+\infty)$.

\vspace{0.2cm}
Let us also  define,  for a fixed $c \in ]0,1[$, the so-called {\it{de Branges function}}:
\begin{equation}\label{fonctionE}
 E(z,c) = s_0(c,z) +i s_0^{[1]}(c,z).
\end{equation}
Since $\tau$ is in the limit circle case at both points $x=0$ and $x=1$, the de Branges function $E(z,c)$ belongs to the Cartwright class, i.e. the function $z \rightarrow E(z,c)$ is entire of exponential type and satisfies the growth condition:
\begin{equation}\label{cartwright}
\int_{\R} \frac {\log_+ (|E(\lambda, c) |)}{1 + \lambda^2} \ d\lambda < \infty,
\end{equation}
where $\log_+ = \max \ (\log, 0)$ is the positive part of the logarithm, (see \cite{EGNT2}, Lemma 4.2). Thus, by a theorem of Krein (\cite{RR}, Theorem 6.17), the de Branges function is of bounded type in the open upper and lower complex half-plane, (i.e it can be written as the quotient of two bounded analytic functions).

\vspace{0.2cm}
Now, we are able to proceed to the main inverse uniqueness result for the spectral measure. Let $S$, (resp. $\tilde{S}$), be the self-adjoint Dirichlet realization of a Sturm-Liouville
differential expressions $\tau$, (resp. $\tilde{\tau}$), on the interval $[0,1]$. We shall use the additional subscript $\tilde{}$ for all quantities corresponding to $\tilde{S}$.

\vspace{0.2cm}\noindent
Since the de Branges functions are of bounded type, we have the following uniqueness result, which follows immediately from (\cite{EGNT2}, Theorem 3.4)  :

\vspace{0.1cm}
\begin{thm} \label{ISM}
Under the hypotheses \ref{H}, assume that the spectral measures $\rho$, $\tilde{\rho}$ associated to $S$ and $\tilde{S}$ respectively are equal. Then there is an AC bijection $\eta$ from $[0,1]$ onto $[0,1]$ and an AC positive function $\kappa$ on $[0,1]$ such that $p \kappa'$ is AC on $[0,1]$ and
\begin{eqnarray*}
 \eta' \tilde{r} \circ \eta &=& \kappa^2 r,   \\
 \tilde{p} \circ \eta &=& \eta' \kappa^2 p, \\
 \eta' \tilde{q} \circ \eta &=& \kappa^2 q - \kappa \ (p \kappa')'.
\end{eqnarray*}
Moreover, the map $V : L^2 ((0,1); \tilde{r} (x) dx) \rightarrow L^2 ((0,1); r (x) dx)$ given by
\begin{equation}
 (Vf)(x) = \kappa (x) f(\eta(x)),
\end{equation}
is unitary and we have $S = V \tilde{S} V^{-1}$.
\end{thm}

\vspace{0.2cm}
As a by-product, in the particular  case where $p=r$, $\tilde{p} = \tilde{r}$, we have the simpler uniqueness result:

\vspace{0.1cm}

\begin{coro}\label{PR}
Assume hypotheses \ref{H} hold with  $p=r$, $\tilde{p} = \tilde{r}$, and that $M(z) = \tilde{M}(z)$ for all $z \in \C^+$. Then there is an absolutely continuous positive function $\kappa$ on $(0,1)$ such that $p \kappa'$ is absolutely continuous on $[0,1]$, satisfying $\kappa(0)=1$, $p\kappa' (0)=0$ and
\begin{eqnarray}
 \tilde{p}  &=& \kappa^2 p, \label{syst1} \\
 \tilde{q} &=& \kappa^2 q - \kappa \ (p \kappa')'. \label{syst3}
\end{eqnarray}
Moreover, the map $V : L^2 ((0,1); \tilde{p} (x) dx) \rightarrow L^2 ((0,1); p (x) dx)$ given by
\begin{equation}
 (Vf)(x) = \kappa (x) f(x),
\end{equation}
is unitary and we have $S = V \tilde{S} V^{-1}$.
\end{coro}

\begin{proof}
Assume that $M(z)= \tilde{M}(z) $  for all $z \in \C^+$. So, thanks to (\ref{mesure}), the spectral measures associated to $S$ and $\tilde{S}$ coincide, and we can use Theorem \ref{ISM}. Since $\eta$ is an absolutely continuous bijection, its derivative has a constant sign, which is positive thanks to Theorem \ref{ISM}. Then, we easily deduce that $\eta' =1$ almost everywhere, which implies $\eta(x)=x$ for all $x \in[0,1]$.
\vspace{0.1cm}

Now, let us prove that $\kappa(0)=1$, $p\kappa' (0)=0$. Recalling that  $S = V \tilde{S} V^{-1}$, we see that the functions $V \tilde{c}_0$ and $V \tilde{s}_0$ satisfy the equation $(\tau-z)u=0$. In particular, since $(c_0, s_0)$ is a (FSS), we have:
\begin{eqnarray*}
 \kappa \tilde{c}_0 &=& a c_0 + b s_0, \\
 \kappa \tilde{s}_0 &=& c c_0 + d s_0,
 \end{eqnarray*}
for some constants $a,b,c$ and $d$. Using (\ref{FSS}), we obtain immediately
\begin{eqnarray}
 \kappa \tilde{c}_0 &=& \kappa(0) c_0 + (p\kappa')(0) s_0, \label{link1} \\
 \kappa \tilde{s}_0 &=&  \frac{1}{\kappa(0)} s_0. \label{link10}
\end{eqnarray}
From the equality $M(z)=\tilde{M}(z)$, we deduce from (\ref{WT}):
\begin{equation}\label{intermediaire}
c_0(1,z) \ \tilde{s}_0(1,z) = \tilde{c}_0(1,z)\  s_0(1,z).
\end{equation}
So, multiplying (\ref{intermediaire}) by $\kappa(1)$, and using (\ref{link1}), (\ref{link10}) evaluated at $x=1$, as well as $s_0(1,z)\not=0$, (since $z \in \C^+$ does not belong to the spectrum of $S$), we have:
\begin{equation}\label{link2}
 (1-\kappa²(0)) c_0(1,z) = \kappa (0) (p\kappa')(0) s_0 (1,z).
\end{equation}
Now, if $\kappa(0) \not=1$, thanks to (\ref{WT}) again and (\ref{link2}), we should have:
\begin{equation}
M(z) = \frac{\kappa(0)}{\kappa^2(0)-1} (p\kappa')(0) \ ,\ {\rm{for \ all\ }} z \in \C^+.
\end{equation}
We deduce by (\ref{mesure}) that the spectral measure associated to $S$ is identically zero, which is impossible thanks to  (\ref{Funitary}).
We conclude that $\kappa(0)=1$, and thanks to (\ref{link2}), we have $(p\kappa')(0)=0$.
\end{proof}

\Section{The anisotropic Calder\'{o}n problem for doubly warped product metrics in the $L^\infty$ setting.} \label{3}

\subsection{The Dirichlet-to-Neumann map.}

Let us recall that we consider a Riemannian manifold $(M,g)$ with boundary, having the topology of a cylinder given by $M = [0,1] \times K_1 \times K_2$, where each factor $(K_j, g_j)$ is a closed $n_j$-dimensional smooth Riemannian manifold. We assume that the manifold $M$  is equipped with a Riemannian metric
\begin{equation} \label{Metric}
  g = h_1(x) dx^2 + h_1(x) g_1 + h_2(x) g_2,
\end{equation}
where $h_1, \ h_2$ are bounded measurable positive functions depending only on the variable $x \in [0,1]$. To simplify the notation, we set :
\begin{equation}\label{defh}
h= h_1^{n_1-1} h_2^{n_2}.
\end{equation}
The positive Laplace-Beltrami operator can be expressed in our coordinates system as:
\begin{equation}\label{laplacien}
  -\Delta_g = -\frac{1}{h_1 \sqrt{h}} \partial_x (\sqrt{h} \partial_x)  - \frac{1}{h_1} \Delta_{g_1} - \frac{1}{h_2} \Delta_{g_2},
\end{equation}
where $-\Delta_{g_1}$ and $-\Delta_{g_2}$ are the positive Laplace-Beltrami operators on $(K_1,g_1)$ and $(K_2,g_2)$ respectively.

\vspace{0.2cm}\noindent
We assume  there exist two constants $0 < c < C$ such that $c \leq h_j(x) \leq C$ almost everywhere in $[0,1]$, which ensures  that  $-\Delta_g$ is uniformly elliptic on $M$.

\vspace{0.2cm}
We look at the Dirichlet problem at a fixed frequency $\lambda$ on $M$ such that $\lambda\notin \{ \lambda_j\}_{j \geq 1}$ where $\{ \lambda_j\}_{j \geq 1}$ is the Dirichlet spectrum of the operator $-\Delta_g$:
\begin{equation} \label{Eq00}
  \left\{ \begin{array}{cc} -\Delta_g u = \lambda u, & \textrm{on} \ M, \\ u = \psi, & \textrm{on} \ \partial M. \end{array} \right.
\end{equation}
It is well known (see for instance \cite{GT}, Theorem 8.3) that, for any $\psi \in H^{1/2}(\partial M)$, there exists a unique weak solution $u \in H^1(M)$ of (\ref{Eq00}).

\vspace{0.2cm}
Following the approach given in \cite{Sa}, and recalled in the Introduction, we can define the Dirichlet-to-Neumann map (DN map) in a weak sense. For $\phi \in H^{\half}(\partial M)$, let $v$ be any function in $H^1(M)$ such that $v_{|\partial M} = \phi$.
The DN map is then defined  as the operator $\Lambda_g$ from $H^{\half} (\partial M)$ to $H^{-\half}(\partial M)$ {\it{via the bilinear form}}
\begin{equation}\label{DNW}
\langle \Lambda_g(\lambda) \psi, \phi \rangle = \int_M \left( \langle du, dv \rangle_g - \lambda u v \right) \ dVol_g\,,
\end{equation}
where the Riemannian inner product  $\langle \omega, \eta \rangle_g$ of the $1-$forms $\omega,  \eta$, is defined  as:
$$
\langle \omega, \eta \rangle_g = g^{ij} \omega_i \eta_j,
$$
and where $dV = h_1^{\frac{n_1 +1}{2}} h_2^{\frac{n_2}{2}} \ dx\ dV_1 \ dV_2$ is the Riemannian volume form on the manifold $(M,g)$.

\vspace{0.2cm}\noindent
In our setting, with the obvious notation, we have:
\begin{equation}
\langle du, dv \rangle_g = \frac{1}{h_1}  \partial_x u \ \partial_x v + \frac{1}{h_1} \langle du, dv \rangle_{g_1} + \frac{1}{h_2} \langle du, dv \rangle_{g_2}.
\end{equation}
In particular, for any {\it{sufficiently regular}} functions $u,v$, and using the Green's formula on each closed smooth manifold $K_j$, we obtain immediately from (\ref{DNW}):
\begin{equation}\label{DNW1}
\langle \Lambda_g(\lambda) \psi, \phi \rangle = \int_M  \left( \frac{1}{h_1}  \partial_x u \ \partial_x v +
[ \frac{1}{h_1}  (-\Delta_{g_1}v) + \frac{1}{h_2} (-\Delta_{g_2}v) - \lambda v] \  u  \  \right) \ dVol_g .
\end{equation}

\vspace{0.5cm}
\noindent
Recall also that the boundary $\partial M$ of $M$ is disconnected and consists in the disjoint union of two copies of $K_1 \times K_2$, that is:
$$
  \partial M = \Gamma_0 \cup \Gamma_1, \quad \Gamma_0 = \{0\} \times K_1 \times K_2, \quad \Gamma_1 = \{1\} \times K_1 \times K_2.
$$
Then, we can decompose the Sobolev spaces $H^s(\partial M)$ as $H^s(\partial M) = H^s(\Gamma_0) \oplus H^s(\Gamma_1)$ for any $s \in \R$ and we use the vector notation
$$
 \varphi = \left( \begin{array}{c} \varphi^0 \\ \varphi^1 \end{array} \right),
$$
for all elements $\varphi$ of $H^s(\partial M) = H^s(\Gamma_0) \oplus H^s(\Gamma_1)$. It follows that the DN map, which is a linear operator from $H^{1/2}(\partial M)$ to $H^{-1/2}(\partial M)$, has the structure of an operator valued $2 \times 2$ matrix
$$
  \Lambda_g(\lambda) = \left( \begin{array}{cc} L(\lambda) & T_R(\lambda) \\ T_L(\lambda) & R(\lambda) \end{array} \right),
$$
where $L(\lambda), R(\lambda), T_R(\lambda), T_L(\lambda)$ are operators from $H^{1/2}(K_1 \times K_2)$ to $H^{-1/2}(K_1 \times K_2)$. Moreover, the components of this matrix are nothing but the partial DN map defined in the Introduction:
\begin{equation}\label{DNmap1}
  L(\lambda) = \Lambda_{g,\Gamma_0, \Gamma_0}(\lambda), \quad R(\lambda) = \Lambda_{g,\Gamma_1, \Gamma_1}(\lambda),
\end{equation}
\begin{equation}\label{DNmap2}	
	T_L(\lambda) = \Lambda_{g,\Gamma_0, \Gamma_1}(\lambda), \quad T_R(\lambda) = \Lambda_{g,\Gamma_1, \Gamma_0}(\lambda).
\end{equation}

\subsection{The separation of variables.}

\vspace{0.2cm}
The cylindrical metric structure of the Riemannian manifold $(M,g)$ is such that one can find a simple expression of the solution $u$ of the Dirichlet problem for the metric $g$. For $j=1,2$,
we introduce the Hilbert basis of real-valued harmonics of the Laplace-Beltrami operator $-\Delta_{g_j}$ on $L^2(K_j)$:
\begin{equation}\label{decomposition}
-\Delta_{g_1} \Phi_{m} = \m \Phi_{m} \ \ ,\ \ -\Delta_{g_2} \Psi_{n} = \n \Psi_{n}.
\end{equation}
The eigenvalues $\m$ and $\n$ are ordered (counting multiplicities) according to
$$
0=\mu_{0} \leq \mu_{1} \leq ... \leq \mu_{m} \leq \ ...  \ \ ,\ \  0=\nu_{0} \leq \nu_{1} \leq ... \leq \nu_{n} \leq \ ...
$$
We recall that the Weyl law implies the following asymptotics on the eigenvalues $\m$ and $\n$. For instance, we have:
\begin{equation} \label{weyl}
 \m \sim \frac{\sqrt{2\pi}}{(\omega_1 \ vol(K_1))^{\frac{2}{n_1}}} \  m^{\frac{2}{n_1}} \ ,\ m \rightarrow + \infty,
\end{equation}
where $\omega_1$ is the volume of the unit ball of $\R^{n_1}$ and a similar expression holds for $\n$.

\vspace{0.5cm}
Now, let us take advantage of this separation of variables to express the DN map differently. First, we write the boundary data $\psi = (\psi^0, \psi^1) \in H^{1/2}(\Gamma_0) \times H^{1/2}(\Gamma_1)$, using their Fourier series decomposition as
$$
 \psi^0 = \sum_{m,n \in \N} \psi^0_{mn} Y_{mn}, \quad \psi^1 = \sum_{m,n \in \N} \psi^1_{mn} Y_{mn},
$$
where we have set
$$
  Y_{mn} = \Phi_{m}\  \Psi_{n}.
$$
Thus, the unique solution $u$ of (\ref{Eq00}) can be sought in the form
$$
  u = \sum_{m,n \in \N} u_{mn}(x) \  Y_{mn}.
$$
Clearly, thanks to (\ref{laplacien}), for all $m,n \in \N$, the function $u_{mn}$ is the unique solution of a singular Sturm-Liouville equation  (w.r.t. $x$) on $[0,1]$ with boundary conditions, given by
\begin{equation} \label{Eq2}
  \left\{ \begin{array}{c} - \frac{1}{\sqrt{h}} \left( \sqrt{h}u_{mn}'\right)'  + (  \n \frac{h_1}{h_2}  - \lambda h_1 ) u_{mn} = -\m u_{mn}, \\
    u_{mn}(0) = \psi^0_{mn}, \quad u_{mn}(1) = \psi^1_{mn}. \end{array} \right.
\end{equation}

\vspace{0.2cm} \noindent
We emphasize that the equation (\ref{Eq2}) fits into the framework of Section 2 using the dictionary
\begin{equation}\label{dico}
p=r=\sqrt{h} \ ,\ q= q_{\lambda,n}= ( \n \frac{h_1}{h_2}  - \lambda h_1 ) \sqrt{h}\ ,\  z =-\m.
\end{equation}


\vspace{0.2cm} \noindent
As in \cite{DKN2}, the DN map is now diagonalized on the Hilbert basis $\{ Y_{mn} \}_{m,n \in \N}$ and is shown to have a very simple expression on each harmonic.
Indeed, we choose  $\psi, \ \phi \in C^{\infty}(\partial M)$ in the form:
$$
\psi =  \left( \begin{array}{c} \psi^0_{mn} \\ \psi^1_{mn} \end{array} \right) \otimes Y_{mn} \ ,\
\phi =  \left( \begin{array}{c} \phi^0_{mn} \\ \phi^1_{mn} \end{array} \right) \otimes Y_{mn}.
$$
By construction, the solution $u_{mn}$ of (\ref{Eq00}) with the boundary condition $\psi$ satisfies also (\ref{Eq2}). Now, in (\ref{DNW1}), we take $v = v_{mn}(x) \   Y_{mn} \in C^{\infty}(M)$  such that $v_{mn}(0) = \phi^0_{mn}$ and $v_{mn}(1) = \phi^1_{mn}$. In other words, we have $v_{ |\partial M }= \phi$, and clearly
$$
-\Delta_{g_1} v = \m \ v_{mn}(x)\  Y_{mn} \ ,\ -\Delta_{g_2} v = \n \  v_{mn}(x)\  Y_{mn} \ ,\ \partial_x v= v_{mn}'(x)\  Y_{mn}.
$$
Substituting these expressions in (\ref{DNW1}), since the volume form  is given by 
$$
dV = h_1^{\frac{n_1 +1}{2}} h_2^{\frac{n_2}{2}} \ dx\ dV_1 \ dV_2\,,
$$ 
we obtain:
\begin{equation*}
\langle \Lambda_g(\lambda) \psi, \phi \rangle = \int_0^1  \left( \frac{1}{h_1}  u_{mn}' \ v_{mn}' +
[  \frac{\m}{h_1}  +  \frac{\n}{h_2}  - \lambda ] \  v_{mn} u_{mn} \  \right) \ h_1^{\frac{n_1 +1}{2}} h_2^{\frac{n_2}{2}} \ dx.
\end{equation*}
Recalling that we have set $h = h_1^{n_1-1} h_2^{n_2}$, we obtain:
\begin{equation*}
\langle \Lambda_g(\lambda) \psi, \phi \rangle = \int_0^1  \sqrt{h}   u_{mn}' \ v_{mn}' dx +
\int_0^1 [  \m +  \n \frac{h_1}{h_2}  - \lambda h_1 ] \ \sqrt{h}  \ v_{mn} u_{mn}  \ dx .
\end{equation*}
Now, using the fact that $\sqrt{h}u_{mn}'$ is absolutely continuous on $[0,1]$, we can integrate by parts the first of the above integrals:
\begin{equation}
\int_0^1  \sqrt{h}  u_{mn}' \ v_{mn}' dx = (\sqrt{h} u_{mn}')(1) \phi_{m}^1 - (\sqrt{h} u_{mn}')(0) \phi_{m}^0 -  \int_0^1  (\sqrt{h}u_{mn}')' v_{mn} \ dx.
\end{equation}
Thus, using (\ref{Eq2}), we have obtained:
\begin{equation} \label{DNW2}
\langle \Lambda_g(\lambda) \psi, \phi \rangle = (\sqrt{h} u_{mn}')(1) \phi_{m}^1 - (\sqrt{h} u_{mn}')(0) \phi_{m}^0
\end{equation}

\vspace{0.2cm}\noindent
In other words, if we denote
\begin{equation}\label{DNW3}
  \Lambda_g(\lambda)_{|<Y_{mn}>} = \Lambda^{mn}_g(\lambda) = \left( \begin{array}{cc} L^{mn}(\lambda) & T^{mn}_R(\lambda) \\ T^{mn}_L(\lambda) & R^{mn}(\lambda) \end{array} \right),
\end{equation}
the restriction of the global DN map to each harmonic $<Y_{mn}>$, we see that this operator has the structure of a $2 \times2$ matrix and satisfies for all $m,n \in \N$:
\begin{equation} \label{DN-Partiel-0}
  \Lambda_g^{mn}(\lambda) \left( \begin{array}{c} \psi^0_{mn}  \\ \psi^1_{mn}  \end{array} \right) \otimes Y_{mn} =
  \left( \begin{array}{c} - (\sqrt{h} u_{mn}')(0)  \\ \ \ (\sqrt{h} u_{mn}')(1)  \end{array} \right) \otimes Y_{mn}.
\end{equation}

\vspace{0.5cm}
\noindent
As in \cite{DKN2}, we can further simplify the partial DN maps $\Lambda^{mn}_g(\lambda)$ by interpreting their coefficients as the characteristic and Weyl-Titchmarsh functions of the ODE (\ref{Eq2}) with appropriate boundary conditions. We recall briefly the procedure. First fix $n \in \N$ and consider the ODE
\begin{equation} \label{Eq3}
  -\frac{1}{\sqrt{h}} \left( (\sqrt{h} v')' + q_{\lambda n}(x) v \right)  = z v, \quad \quad q_{\lambda n} = (\n \frac{h_1}{h_2} - \lambda  h_1 )\sqrt{h}, \quad z = -\m
\end{equation}
with boundary conditions
\begin{equation} \label{BC4-3D}
  v(0) = 0, \quad v(1) = 0.
\end{equation}

Note that the equation (\ref{Eq3}) is nothing but equation (\ref{Eq2}) in which the parameter $ -\m$ is written as $z$ and is interpreted as the spectral parameter of the equation, and where the boundary conditions  have been replaced by Dirichlet boundary conditions at $x = 0$ and $x=1$. Thanks to the results recalled in Section 2, we can define for all $n \in \N$ and all $z \in \C$ the fundamental systems of solutions

$$
  \{ c_0(x,z,\n), s_0(x,z,\n)\}, \quad \{ c_1(x,z,\n), s_1(x,z,\n)\},
$$
by imposing the Cauchy conditions
\begin{equation} \label{FSS1}
  \left\{ \begin{array}{cccc} c_0(0,z,\n) = 1, & c_0^{[1]}(0,z,\n) = 0, & s_0(0,z,\n) = 0, & s_0^{[1]}(0,z,\n) = 1, \\
 	                  c_1(1,z,\n) = 1, & c_1^{[1]}(1,z,\n) = 0, & s_1(1,z,\n) = 0, & s_1^{[1]}(1,z,\n) = 1. \end{array} \right.
\end{equation}

\vspace{0.5cm}
We now come back to the expression (\ref{DN-Partiel-0}) of the partial DN map $\Lambda^{m}_g(\lambda)$. For all $m,n \in \N$, we need to express $(\sqrt{h}u_{mn}')(0)$ and $(\sqrt{h}u_{mn}')(1)$ in terms of $\psi^0_{mn}$ and $\psi^1_{mn}$ in order to find the expressions of the coefficients $L^{mn}(\lambda), T^{mn}_R(\lambda), T^{mn}_L(\lambda), R^{mn}(\lambda)$. But the solution $u_{mn}$ of (\ref{Eq2}) can be written as linear combinations of the (FSS)
$$
  \{ c_0(x,z,\n), s_0(x,z,\n)\}, \quad \quad \{ c_1(x,z,\n), s_1(x,z ,\n)\}.
$$	
More precisely, we write
\begin{equation}\label{linear}
  u_{mn} = \alpha \,c_0  + \beta \,s_0 = \gamma \,c_1 + \delta \,s_1,
\end{equation}
for some constants $\alpha,\beta,\gamma,\delta$. Using (\ref{Eq2}) and (\ref{FSS1}), we first get
\begin{equation} \label{a1}
  \left( \begin{array}{c} u_{mn}(0) \\ u_{mn}(1) \end{array} \right) = \left( \begin{array}{c}  \psi^0_{mn} \\  \psi^1_{mn} \end{array} \right) = \left( \begin{array}{c} \alpha \\ \gamma \end{array} \right) = \left( \begin{array}{c} \gamma \,c_1(0,z,\n) + \delta \, s_1(0,z,\n),  \\ \alpha \,c_0(1,z,\n) + \beta \,s_0(1,z,\n) \end{array} \right).
\end{equation}
We deduce that
\begin{equation} \label{a2}
  \left( \begin{array}{c} \beta \\ \delta  \end{array} \right) = \left( \begin{array}{cc} -\frac{c_0(1,z,\n)}{s_0(1,z,\n)} & \frac{1}{s_0(1,z,\n)} \\ \frac{1}{s_1(0,z,\n)} & -\frac{c_1(0,z,\n)}{s_1(0,z,\n)} \end{array} \right) \left( \begin{array}{c} \psi^0_{mn} \\ \psi^1_{mn}  \end{array} \right).
\end{equation}
Calculating the quasi-derivatives in the equation (\ref{linear}), we obtain immediately:
\begin{equation} \label{calcul}
\beta = (\sqrt{h}u_{mn}')(0) \ ,\ \delta = (\sqrt{h}u_{mn}')(1).
\end{equation}
Then, it follows from (\ref{DN-Partiel-0}), (\ref{a2}) and (\ref{calcul}):
\begin{equation} \label{a3}
 \Lambda_g^{mn} (\lambda)  = \left( \begin{array}{cc} \frac{c_0(1,z,\n)}{s_0(1,z,\n)} & -\frac{1}{s_0(1,z,\n)} \\ \frac{1}{s_1(0,z,\n)} & -\frac{c_1(0,z,\n)}{s_1(0,\m,\n)} \end{array} \right)
\end{equation}

\vspace{0.2cm}\noindent
Finally, using the definition of the characteristic and the Weyl-Titchmarsh functions introduced in Section 2, we easily show the DN map on each harmonic can be written as:
\begin{equation} \label{DNpartial}
 \Lambda_g^{mn} (\lambda)  = \left( \begin{array}{cc} -M(\m,\n) & -\frac{1}{\Delta(\m,\n)} \\ -\frac{1}{\Delta(\m,\n)} & -N(\m,\n) \end{array} \right)
\end{equation}

\subsection{The complex angular momentum method.} \label{CAM}

In this section, we set up the complex angular momentum (CAM) method we have already introduced in \cite{DKN1, DKN2, DKN3, DKN4, DN3, DN4, DGN}. More precisely, we allow the angular momentum $z= -\m$ appearing in all the previous expressions of the DN map to be a complex number, and we shall take account this new amount of information to solve the Calder\'on problem.

\vspace{0.2cm}\noindent
So, let us consider two metrics $g$ and $\tilde{g}$ on $M=[0,1] \times K_1 \times K_2$ in the form
$$
g = h_1(x) dx^2 + h_1 (x) g_1 + h_2 (x) g_2\ ,\ \tilde{g} = \tilde{h}_1(x) dx^2 + \tilde{h}_1 (x) g_1 + \tilde{h}_2 (x) g_2 .
$$
As usually, we add  the subscript $\, \tilde{} \,$ to all the quantities referring to the metric $\tilde{g}$. If $\n$ is a fixed eigenvalue, we assume that
\begin{equation}\label{egalityM}
  M(\m, \n) = \tilde{M}(\m, \n), \quad \forall m \geq 0.
\end{equation}
Using (\ref{MD}), this implies
$$
  D(\m,\n)\tilde{\Delta}(\m,\n) - \tilde{D}(\m,\n) \Delta(\m,\n)= 0, \quad \forall m \geq 0.
$$
Now, let us introduce for each $n \geq 0$ the function
\begin{equation}\label{defF}
F(z) = D(-z^2,\n)\tilde{\Delta}(-z^2,\n) - \tilde{D}(-z^2,\n) \Delta(-z^2,\n).
\end{equation}
From the analytic properties of the (FSS) $\{c_0, s_0\}$ and $\{c_1, s_1\}$ established in Section 2, we see that $F$ is an entire function of exponential type (\textit{i.e.} $\forall z \in \C, \quad |F(z)| \leq C e^{A|z|}$ for some positive constants $A, C$). Moreover $F$ vanishes on the sequence $(\sqrt{\m})_{m \geq 0}$.

\vspace{0.2cm}\noindent
We aim to prove that $F$ vanishes identically on $\C$. To show this, we use the Duffin-Schaeffer Theorem (see \cite{Boa}, Thm 10.5.3):

\begin{thm} \label{Duffin}
Let $f$ be an entire function whose indicator function satisfies :
\begin{equation} \label{3-DSHyp}
  h_f(\theta) \leq a|\cos (\theta)| + b |\sin(\theta)|, \quad b < \pi.
\end{equation}
Assume that $(\lambda_k)_{k \geq 1}$ is an increasing sequence of real numbers satisfying
\begin{equation} \label{3-Separation}
  |\lambda_{k+1} - \lambda_k|\geq \delta > 0, \quad |\lambda_k - k| \leq L, \quad |f(\lambda_k)| \leq M.
\end{equation}
Then there exists $K$ such that
$$
  |f(x)| \leq K M, \quad \forall x \in \R.
$$
In particular, if $f(\lambda_k) = 0, \ \forall k \geq 1$, then $f(z) = 0, \ \forall z \in \C$.
\end{thm}

Note first that $F(z) = d(iz,\n) \tilde{\delta}(iz,\n) - \tilde{d}(iz,\n) \delta(iz,\n)$ where we have used the notations introduced in (\ref{2-Nota}). Using Corollary \ref{2-EstRef} and the usual properties of indicator functions (see \cite{Lev}, Lecture 8), we get immediately
\begin{equation} \label{3-IndicatorF}
  h_F(\theta) := \limsup_{r \to \infty} \frac{\log |F(re^{i\theta})|}{r} \leq  2|\cos \theta|.
\end{equation}
Hence the function $F$ satisfies the assumption (\ref{3-DSHyp}) of the Duffin-Schaeffer Theorem. However, the sequence $(\sqrt{\m})_{m \geq 0}$ need not in general satisfy the separation assumptions (\ref{3-Separation}). To overcome this difficulty, we proceed as follows.

We remark secondly that the Weyl law implies the following asymptotic on the eigenvalues $\sqrt{\m}$ (repeated according multiplicity):
\begin{equation} \label{3-Weyl}
  \sqrt{\m} =  C m^{\frac{1}{n_1}} + O(1),
\end{equation}
where $C$ denotes a suitable constant independent of $m$ (see for instance \cite{SaVa}). Let us then set $G(z) := F(C z)$. Then $G$ is still an entire function satisfying (\ref{3-DSHyp}) and it is clear that $G$ vanishes on the subsequence of $\frac{1}{C} \sqrt{\m}$ defined by $\eta_m := \frac{1}{C} \sqrt{\mu_{m^{n_1}}}$. Thanks to (\ref{3-Weyl}), this subsequence satisfies $\eta_m = m + O(1)$ and thus verifies the second assumption in (\ref{3-Separation}), but still not the first one.

The last step consists in introducing the function $L(z) := G(Nz)$ for some $N \in \N$. Then the function $L$ is once again an entire function satisfying (\ref{3-DSHyp}) and vanishes on the subsequence $\lambda_m = \frac{1}{N} \eta_{Nm}$ which satisfies $\lambda_m = m + \frac{O(1)}{N}$. We conclude that for $N$ large enough, the function $L$ and the sequence $\lambda_m$ satisfy the assumptions of the Duffin-Schaeffer Theorem. As a consequence, we get
$$
  L(z) = 0, \quad \forall z \in \C,
$$
or equivalently
$$
  F(z) = 0, \quad \forall z \in \C.
$$
In particular, using (\ref{WT}) and (\ref{defF}) again, we have proved under the assumption (\ref{egalityM}) that:
\begin{equation}\label{egalityMonC}
  M(z, \n) = \tilde{M}(z, \n), \quad \forall z \in \C.
\end{equation}

\subsection{Proof of Theorem \ref{MainThm}} \label{UniquenessCalderon}

As a by-product of (\ref{egalityMonC}) and Corollary \ref{PR} with $p=r =\sqrt{h}$ and $\tilde{p}=\tilde{r}=\sqrt{\tilde{h}}$, we obtain the following result:

\begin{coro}
For a fixed eigenvalue $\n$, assume that the Weyl-Titchmarsh functions $M(. , \n)$ and $\tilde{M}(., \n)$  satisfy:
\begin{equation}\label{egalityM1}
  M(\m, \n) = \tilde{M}(\m, \n), \quad \forall m \geq 0.
\end{equation}
Then there is an absolutely continuous positive function $\kappa$ on $[0,1]$ such that $\sqrt{h} \kappa'$ is absolutely contiuous on $[0,1]$, with $\kappa(0)=1, (\sqrt{h}\kappa') (0)=1$, and
\begin{eqnarray}
 \tilde{h}  &=& \kappa^4 h, \label{M1-1}\\
 \tilde{q}_{\lambda,n} &=& \kappa^2 q_{\lambda, n} - \kappa (\sqrt{h}\kappa')'. \label{M1-2}
\end{eqnarray}
\end{coro}

\vspace{0.5cm}\noindent
Now,  recalling that the potentials $q_{\lambda,n}$ and $\tilde{q}_{\lambda,n}$ are given by (\ref{dico}), we can write  (\ref{M1-2}) as:
\begin{equation} \label{equalityq}
\left( \n \frac{\tilde{h}_1}{\tilde{h}_2} - \lambda \tilde{h}_1 \right) \sqrt{\tilde{h}} = \kappa^2 \left( \n \frac{h_1}{h_2} - \lambda h_1 \right) \sqrt{h}
- \kappa (\sqrt{h}\kappa')'.
\end{equation}

\vspace{0.2cm} \noindent
In particular, if we assume that (\ref{equalityq}) holds for two different eigenvalues  $\n$, we deduce:
\begin{eqnarray}
\frac{\tilde{h}_1 \sqrt{\tilde{h}}} {\tilde{h}_2} &=& \kappa^2 \ \frac{h_1 \sqrt{h}}{h_2}, \label{equalityh1} \\
\lambda \tilde{h}_1 \sqrt{\tilde{h}} &=& \lambda \kappa^2 h_1 \sqrt{h} + \kappa (\sqrt{h}\kappa')'. \label{equalityh2}
\end{eqnarray}
Thus, (\ref{M1-1}) and (\ref{equalityh1}) imply:
\begin{equation}\label{relation1}
\frac{\tilde{h}_1} {h_1} = \frac{\tilde{h}_2}{h_2}.
\end{equation}
So, using (\ref{defh}) and (\ref{M1-1}) again,  we obtain:
\begin{eqnarray}
\tilde{h}_1 &=& \kappa^{\frac{4}{n_1+n_2-1}} \ h_1, \label{conformalfactor}\\
\tilde{h}_2 &=& \kappa^{\frac{4}{n_1+n_2-1}} \ h_2.
\end{eqnarray}
In the same way, thanks to (\ref{M1-1}), (\ref{equalityh2}) and (\ref{conformalfactor}), we see that $\kappa$ satisfies:
\begin{equation}
(\sqrt{h} \kappa)' - \lambda \   ( \kappa^{\frac{4}{n_1+n_2-1}+1} -\kappa) \sqrt{h} h_1 =0.
\end{equation}

\vspace{0.2cm}\noindent
So, at this stage, we have proved the following result :

\vspace{0.2cm}

\begin{coro} \label{twomodes}
Assume that for two different eigenvalues $\n$, we have:
\begin{equation}
  M(\m, \n) = \tilde{M}(\m, \n), \quad \forall m \geq 0.
\end{equation}
Then, there exists an absolutely continuous positive conformal factor $\kappa(x)$ on $[0,1]$ such that:
\begin{equation}\label{conform}
\tilde{g} = \kappa^{\frac{4}{n_1+n_2-1}} (x)  \ g.
\end{equation}
Moreover, $ \sqrt{h} \kappa'$ is absolutely continuous on $[0,1]$ and satisfies the non-linear ordinary differential equation:
\begin{equation}\label{nonlinearEDO}
-(\sqrt{h} \kappa')' + \lambda \ ( \kappa^{\frac{4}{n_1+n_2-1}+1}-\kappa)\sqrt{h}  h_1  = 0  \ \ \ {{almost\ everywhere \ on}} \ [0,1],
\end{equation}
with the boundary conditions $\kappa(0)=1$, $(\sqrt{h}\kappa')(0)=0$.
\end{coro}

\vspace{0.2cm}
\begin{rem}
In \cite{DKN2}, we have studied a similar problem in the case of manifolds corresponding to toric cylinders, but for {\it{smooth}} Riemannian metrics.
We remark that the non-linear ODE (\ref{nonlinearEDO}) is  nothing but the ODE appearing in Lemma 4.1, \cite{DKN2}, stated now in the non-regular case, as well as the Yamabe equation (\ref{Yamabe}) recalled in the Introduction.
\end{rem}

\vspace{0.3cm}\noindent
The non-linear ordinary differential equation (\ref{nonlinearEDO}) enters into the framework of Carath\'eodory's type ODEs which we recall here~\cite{Wa}. Setting
$\nu= \sqrt{h} \kappa'$ and $Y = (\kappa, \nu)$, we can write (\ref{nonlinearEDO}) as $Y' = f(x, Y)$ with:
\begin{eqnarray*}
f : D= [0,1] \times [0,+\infty[ \times \R &\rightarrow& \R^2 \\
        (x,Y)\ \ \ \ \  &\rightarrow& \left( \frac{1}{\sqrt{h}}\  \nu \ , \  \lambda \ ( \kappa^{\frac{4}{n_1+n_1-1}+1}-\kappa)\sqrt{h}  h_1   \right).
\end{eqnarray*}
The function $f(x,Y)$ is required to satisfy Carath\'eodory's conditions  in $D$: for a fixed $x$, $f(x,Y)$ is continuous as a function of $Y$, and measurable as a function of $x$ for a fixed $Y$. Moreover, for every  compact set $K \subset D$, we have, for all $(x,Y) \in K$,
\begin{equation}
||f(x,Y)|| \leq  k(x),
\end{equation}
with $k(x) \in L^1 ([0,1])$. Finally, $f$ satisfies a generalized Lipschitz condition: for all $(x,Y) \in K$,
\begin{equation}
||f(x,Y) - f(x,Z) || \leq l(x) \ || Y-Z||,
\end{equation}
where $l(x) \in L^1 ([0,1])$. So using (Theorem XX, p. 122, \cite{Wa}), we see that (\ref{nonlinearEDO}) has a unique solution in $D$ which is obviously $(\kappa, \nu) = (1, 0)$.

\vspace{0.2cm}
As a conclusion, under the hypotheses of Corollary \ref{twomodes}, we have obtain that $\tilde{g} = g$. Then, Theorem \ref{MainThm} follows easily from (\ref{DNmap1}), (\ref{DNW3}) and  (\ref{DNpartial}).

\begin{rem} \label{Extension-MainThm}
Let $h_0, h_1, h_2$ be measurable functions on $[0,1]$ such that
\begin{enumerate}
\item	$h_j > 0, \ j=0,1,2$ a.e. $x \in [0,1]$,
\item there exist two constants $c, C >0$ such that $c \leq h_1(x), \ h_2(x) \leq C$ a.e. in $[0,1]$,
\item $\sqrt{h_0} \in L^1(0,1)$.
\end{enumerate}
Consider on $M = [0,1] \times K_1 \times K_2$ the class of Riemannian metrics
\begin{equation} \label{ExtendedMetric}
  g = h_0(x) dx^2 + h_1(x) g_1 + h_2(x) g_2.  		
\end{equation}

Let $y = \phi(x) := \int_0^x \sqrt{\frac{h_0}{h_1}} ds$. Under our assumptions above, it is clear that $\phi$ is a bijection from $[0,1]$ to $[0,A]$ with $A = \int_0^1 \sqrt{\frac{h_0}{h_1}} ds$, which is an increasing function, such that $\phi$ and $\phi^{-1}$ are absolutely continuous on $[0,1]$ and $[0,A]$ respectively. In this new coordinate system, the manifold is described as the cylinder $\hat{M} = [0,A] \times K_1 \times K_2$ and the metric $g$ takes the form
\begin{equation} \label{OldMetric}
  \hat{g} = (\phi^{-1})^* g = H_1(y) dy^2 + H_1(y) g_1 + H_2(y) g_2,  		
\end{equation}
with $H_j(y) = h_j(x(y)), \ j=1,2$. Since absolutely continuous functions of one variable are admissible transformations for the change of variables in Lebesgue integration (see \cite{Ru}, Theorem 7.26) and since $\phi$ clearly preserves the boundary of $(M,g)$, we easily see that if $u$ is a weak solution of the Dirichlet problem (\ref{Eq00}) on $(M,g)$, then $\tilde{u} = u \circ \phi$ is a solution of the same Dirichlet problem (\ref{Eq00}) for $(\hat{M}, \hat{g})$. From this and the calculations made for the DN maps at the beginning of this section, we deduce easily that
\begin{equation} \label{DN=}
  \Lambda_{(\phi^{-1})^* g}(\lambda) = \Lambda_g(\lambda).
\end{equation}
In other words, the DN map on $(M,g)$ admits the same kind of gauge invariance as in the regular case. Using a slight generalization of Theorem \ref{MainThm} and (\ref{DN=}), we thus have proved

\begin{thm} \label{MainThm-Extended1}
Let $g$ and $\tilde{g}$ be metrics in the class (\ref{ExtendedMetric}) with $h_0, h_1, h_2$ satisfying the above assumptions. Assume that $\Lambda_g(\lambda) = \Lambda_{\tilde{g}}(\lambda)$ for an admissible fixed energy $\lambda$. Then there exists an increasing bijection $\psi: \ [0,1] \longrightarrow [0,1]$ with $\psi$ and $\psi^{-1}$ absolutely continuous on $[0,1]$ such that
$$
  \tilde{g} = \psi^* g.
$$
\end{thm}

\end{rem}



\Section{Uniqueness and invisibility for warped product metrics.} \label{4}

In this section, we re-examine the previous anisotropic Calder\'on problem on a class of Riemannian warped product metrics which are slightly more singular than the ones considered in the previous section. Precisely, we now consider cylinders $M = [0,1] \times K$ equipped with a simpler class of Riemannian metrics given by
\begin{equation} \label{A-Metric}
  g = h_1(x) [ dx^2 + g_K],
\end{equation}
where $h_1 = h_1(x)$ is a measurable positive function a.e. on $[0,1]$ and $(K, g_K)$ is a closed $n$-dimensional Riemannian manifold. We assume in the following that $n \geq 2$. As before, set $h(x) = h_1^{n - 1}(x)$. Roughly speaking, we shall show that there is uniqueness in the anisotropic Calder\'on problem at zero frequency provided the following conditions hold
\begin{equation} \label{A-CondMetric1}
  \sqrt{h} \in L^1(0,1), \quad \frac{1}{\sqrt{h}} \in L^1(0,1).
\end{equation}	
which amounts to the conditions (\ref{MainCond}) given in the Introduction. Let us emphasize that we do not require the metric $g$ to be of uniformly elliptic signature here. The conditions (\ref{A-CondMetric1}) seem to be the minimal conditions under which the Dirichlet problem is \emph{uniquely} solvable, and for boundary data in $H^2(\partial M)$, the unique solution can be shown to belong to the homogeneous Sobolev space $\dot{H}^1(M)$. If we had assumed the extra condition
\begin{equation} \label{A-CondMetric2}
  h_1 \sqrt{h} \in L^1(0,1),
\end{equation}
then the same unique solution of the Dirichlet problem for boundary data in $H^2(\partial M)$ could be shown to be in $H^1(M)$.

Additionally we shall exhibit counter-examples to uniqueness if one of the conditions (\ref{A-CondMetric1}) is not satisfied which shows the sharpness of our results. These counter-examples lead in fact to regions that are invisible to boundary measurements.


\subsection{More singular metrics and uniqueness}

Let us assume the conditions (\ref{A-CondMetric1}) holds. Our first task is to define in a rigorous way the DN map at zero frequency in that case. We shall first prove that the Dirichlet problem
\begin{equation} \label{A-DirichletPb}
  \left\{ \begin{array}{rcl} -\triangle_g u & = & 0, \ \textrm{on} \ M, \\
	                                        u & = & \psi, \ \textrm{on} \ \partial M,
	\end{array}\right. 																				
\end{equation}
possesses a unique solution $u \in \dot{H}^1(M)$ for any boundary data $\psi \in H^2(\partial M)$.

For this we can use the separation of variables described in the previous section and look for distributional solutions $u$ in the form
$$
  u = \sum_{m = 0}^\infty u_{mn}(x) Y_{m},
$$
where $Y_m$ are the normalized eigenfunctions of the Laplace-Beltrami operator on $(K,g_K)$ associated to its eigenvalues $(\m)$, \textit{i.e.}
$$
  -\triangle_{g_K} Y_m = \m Y_m, \quad \forall m \geq 0,
$$
and
$$
  0 = \mu_0 < \mu_1 \leq \mu_2 \leq \dots \leq \m \to \infty.
$$
Here the functions $u_{mn}$ are solutions of the boundary value problem
\begin{equation} \label{A-RadialODE}
  \left\{ \begin{array}{rcl}
  -\frac{1}{\sqrt{h}} \left( \sqrt{h} u_{m}' \right)' & = & z_{m} u_{m}, \\
	u_{m}(0) = \psi^0_{m}, & & u_{m}(1) = \psi^1_{m},
	\end{array} \right.
\end{equation}
where $h(x) = h_1^{n-1}$, $z_{m} = -\mu_m$ and the boundary data $\psi = (\psi^0, \psi^1)$ are represented by their Fourier series expansions $\psi^j = \ds\sum_{m=0}^\infty \psi^j_{m} Y_{m}$.

For the boundary data $\psi= (\psi^0, \psi^1)$, we use the classical Sobolev space of order $s \in \R$ defined by
$$
  H^s(\partial M) = \left\{ \psi = \sum_{m=0}^\infty \psi_{m} Y_{m} \suchthat \sum_{m=0}^\infty (1 + \m)^s \,(|\psi^0_{m}|^2 + |\psi^1_{m}|^2) < \infty \right\}.
$$
Recall also that the homogeneous Sobolev space $\dot{H}^1(M)$ is defined by
$$
  \dot{H}^1(M) = \left\{ u \suchthat  \int_M |\nabla u|^2_g  \, dVol_g < \infty \right\}.
$$
Using separation of variables, we have the following characterization
\begin{equation} \label{A-Hdot1}
  \dot{H}^1(M) = \left\{ u = \sum_{m} u_{m}(x) Y_{m} \suchthat \sum_{m=0}^\infty \int_0^1 (|u'_{m}|^2 + \m |u_{m}|^2) \sqrt{h}(x) \, dx < \infty \right\}.
\end{equation}

Let us prove

\begin{thm} \label{A-UniquenessDP}
For any boundary data $\psi = (\psi^0, \psi^1) \in H^2(\partial M)$, there exists a unique solution $u = \sum_{m = 0}^\infty u_{m}(x) Y_{m}$ in $\dot{H}^1(M)$ of the Dirichlet problem (\ref{A-DirichletPb}). Moreover, the DN map $\Lambda_g$ is well-defined as a bounded operator from $H^s(\partial M)$ to $H^{s-2}(\partial M)$.
\end{thm}

\begin{proof}
It is well known (see \cite{Ze}) that the conditions (\ref{A-CondMetric1}) are the minimal conditions on $h_1$ such that there exist fundamental systems of solutions $\{c_0,s_0\}$ and $\{c_1,s_1\}$ of the ODE (\ref{A-RadialODE}) which satisfy the Cauchy conditions (\ref{FSS}). As in Sections \ref{2} and \ref{3}, we then define the characteristic and WT functions
$$
  \Delta(z) = W(s_0,s_1), \ D(z) = W(c_0,s_1), \ E(z) = W(c_1,s_0), \ M(z) = -\frac{D(z)}{\Delta(z)}, \ N(z) = -\frac{E(z)}{\Delta(z)},
$$
As in Section \ref{2}, we still denote $(\alpha_k)_{k \geq 1}$ the Dirichlet spectrum of the operator $S = -\frac{1}{\sqrt{h}} \frac{d}{dx} \left( \sqrt{h} \frac{d}{dx} \right)$ on $L^2((0,1), \sqrt{h} dx)$. Clearly, the eigenvalues $\alpha_k$ are positive for every $k \geq 1$.     We recall that the Weyl functions are defined for all $z \ne \alpha_k$ by
$$
  \Psi(x,z) = c_0(x,z) + M(z) s_0(x,z), \quad \Phi(x,z) = c_1(x,z) - N(z) s_1(x,z).
$$
A straightforward calculation (see the proof of the Proposition below) shows that the unique solution $u_{m}$ of (\ref{A-RadialODE}) can be expressed as
\begin{equation} \label{A-um}
  u_{m}(x) = \psi^1_{m} \Phi(x,z_{m}) + \psi^0_{m} \Psi(x,z_{m}).
\end{equation}
Moreover, we already showed in Section \ref{3} that the DN map on each harmonic $\langle Y_{m} \rangle$ has the simple expression
\begin{equation} \label{A-DNm}
  \Lambda_g^{m}(\lambda) = \left( \begin{array}{cc} -M(z_{m}) & -\frac{1}{\Delta(z_{m})} \\ -\frac{1}{\Delta(z_{m})} & -N(z_{m}) \end{array} \right).
\end{equation}
We must thus make precise the meaning of the \emph{global} objects 
\begin{equation} \label{u-DN-Global}
  u = \sum_{m= 0}^\infty u_{m}(x) Y_{m}, \quad \Lambda_g(\lambda) = \oplus_{m} \Lambda_g^{m}(\lambda).
\end{equation}

Let us start with some estimates of the WT function $M(z)$. Recall from (\ref{2-SpectralMeasure-M1}), that for all $z \in \C \setminus [\sigma,+\infty)$ with $\sigma < \alpha_1$, we have

\begin{equation}\label{A-SpectralMeasure-M1}
  M(z) = c + \int_\sigma^\infty \left[ \frac{1}{\omega - z} - \frac{\omega}{1+\omega^2} \right] d\rho(\omega), \quad \forall z \in \C^+.
\end{equation}
We now use (\ref{A-SpectralMeasure-M1}) to estimate $M(z)$ when $z$ is negative and large. More precisely, let $z \in (-\infty, \sigma - 1]$ and $\omega \in [\sigma,+\infty)$. Using (\ref{A-SpectralMeasure-M1}) and the estimate
\begin{equation} \label{A-SpectralMeasure-M2}
  \left| \frac{1}{\omega - z} - \frac{\omega}{1+\omega^2} \right| \leq \frac{C_1 |z| + C_2}{1+\omega^2},
\end{equation}
where $C_1$ and $C_2$ are constants, we deduce that
\begin{equation} \label{A-EstM}
  |M(z)| \leq C (1+|z|), \quad \forall z \leq \sigma - 1.
\end{equation}
Recalling that $z_{m} = -\m$, we have proved

\begin{lemma} \label{A-EstMm}
For all $m \geq 0$, we have
$$
|M(z_{m})|, \, |N(z_{m})|  \leq C (1+ \m).
$$
\end{lemma}

We also have

\begin{lemma} \label{A-EstDeltam}
For all $m \geq 0$,
$$
  \left| \frac{1}{\Delta(z_m)} \right| \leq C.
$$
\end{lemma}

\begin{proof}
Since $\alpha_k >0$, it follows from Lemma \ref{2-Hadamard} that $m_1=0$ and for  $z \in \R^-$, $|\Delta(z)| \geq C_1 $.
\end{proof}

Let us finally prove

\begin{prop} \label{A-EstPbDirichlet}
  1) The unique solution $u_{m}$ of (\ref{A-RadialODE}) can be expressed as
$$
  u_{m}(x) = \psi^1_{m} \Phi(x,z_{m}) + \psi^0_{m} \Psi(x,z_{m}).
$$
2) The following estimates hold
$$
  |\Psi(x,z_m)|, \ |\Phi(x,z_m)| \leq 1, \quad \forall x \in [0,1], \ m \geq 0,
$$
$$
  |\sqrt{h} \Psi'(x,z_m)|, \ |\sqrt{h} \Phi'(x,z_m)| \leq C (1+\m), \quad \forall x \in [0,1], \ m \geq 0.
$$
\end{prop}

\begin{proof}
1 - The first assertion follows immediately  from $\Psi(0, z_m)=1$, $\Psi(1,z_m)=0$, $\Phi(0, z_m)=0$ and $\Phi(1,z_m)=1$.
\par\noindent
2 - For instance, let us show that $|\Phi(x,z_m)| \leq 1$  and $|\sqrt{h} \Phi'(x,z_m)| \leq C (1+\m)$ for $m \geq 0$. First, $\sqrt{h}\Phi'(0,z_m) \not=0$, otherwise $\Phi(x,z_m)$ would be equal to $0$ on $(0,1)$.
Let us show that $\sqrt{h}\Phi'(0,z_m) >0$. Indeed, if it is not the case, $\Phi(x,z_m) <0$ for $x$ in a neighborhood of $0$. Now, assume there exists $a \in (0,1)$ such that
$\Phi(x,z_m) <0$ for $x \in (0,a)$ and $\Phi(a,z_m)=0$. Thus, $\sqrt{h}\Phi'(b, z_m)=0$ for some $b \in (0,a)$, and using
\begin{equation}
   \sqrt{h} \Phi'(b, z_m) - \sqrt{h} \Phi'(0, z_m) = - \sqrt{h} \Phi'(0, z_m) = \m \int_0^b \sqrt{h} \Phi(t, z_m) dt,
\end{equation}
we obtain a contradiction. So, we have $\Phi(x, z_m) <0 $ for $x \in (0,1)$ which is not compatible with $\Phi(1,z_m)=1$.
\par
Thus, we have proved that  $\sqrt{h}\Phi'(0,z_m) >0$, and a similar argument implies that $\Phi(x,z_m) >0$ on $(0,1)$. From
\begin{equation}\label{derivee}
  \sqrt{h} \Phi'(x, z_m) = \sqrt{h} \Phi'(0, z_m) +   \m \int_0^x \sqrt{h} \Phi(t, z_m) dt >0,
\end{equation}
we deduce that  $\Phi(x, z_m)$ in an increasing function, which implies  $0 \leq \Phi(x, z_m)\leq 1$ for all $x \in (0,1)$.
\par\noindent
To prove the second estimate,  we remark that  $\sqrt{h}\Phi'(1,z_m)= -N(z_m)$. So, using
\begin{equation}\label{derivee1}
  \sqrt{h} \Phi'(1, z_m)- \sqrt{h} \Phi'(x, z_m) =  \m \int_x^1 \sqrt{h} \Phi(t, z_m) dt,
\end{equation}
and Lemma \ref{A-EstMm}, we obtain immediately $|\sqrt{h} \Phi'(x,z_m)| \leq C (1+\m)$.
\end{proof}
As a corollary, we obtain

\begin{coro} \label{A-Estum}
For all $x \in [0,1]$ and $m \geq 0$,
$$
  |u_m(x)| \leq |\psi_m^0| + |\psi_m^1|, \quad |\sqrt{h}u'_m(x)| \leq C (1+ \m) (|\psi_m^0| + |\psi_m^1|).
$$
\end{coro}

Now the results given in Lemmas \ref{A-EstMm}, \ref{A-EstDeltam}, Proposition \ref{A-EstPbDirichlet} and Corollary \ref{A-Estum} together with (\ref{A-um}) and (\ref{A-DNm}) allow to give a global meaning to (\ref{u-DN-Global}) with the properties stated in the theorem.
\end{proof}

\begin{rem} \label{UniquenessInHdot}
  Note that the extra condition (\ref{A-CondMetric2}) could be used to obtain a solution $u$ of (\ref{A-DirichletPb}) in the usual Hilbert space $H^1(M)$. Indeed, if we look for solutions $u$ in the Sobolev space
$$
  H^1(M) = \left\{ u \suchthat  \int_M (|u|^2 + |\nabla u|^2_g ) \, dVol_g < \infty \right\},
$$
which is equivalent after separation of variables to
\begin{equation} \label{A-H1}
  H^1(M) = \left\{ u = \sum_{m = 0}^\infty u_{m}(x) Y_{m} \suchthat \sum_{m = 0}^\infty \int_0^1 (|u'_{m}|^2 + \m |u_{m}|^2 + h_1 |u_m|^2) \sqrt{h} \,  dx < \infty \right\},
\end{equation}	
then the same proof as above shows that there exists a unique solution $u \in H^1(M)$ of (\ref{A-DirichletPb}) provided $\psi \in H^2(\partial M)$.
\end{rem}

\begin{rem} \label{UsualDN} 
  In the \emph{regular} case, \textit{i.e.} $h_1 \in C^2(0,1)$ and $h_1 > 0$ on $[0,1]$, the WT functions $M$ and $N$ can be shown to satisfy for all $m \geq 0$ the estimates:
$$
  |M(z_{m})|, \, |N(z_{m})|  \leq C (1+ \m)^{1/2}.
$$
From this and the above calculations, we can recover in that case that the DN map $\Lambda_g$ is a well defined bounded operator from $H^s(\partial M)$ into $H^{s-1}(\partial M)$ for any $s \in \R$.  
\end{rem}

According to Theorem \ref{A-UniquenessDP} and Remark \ref{UniquenessInHdot}, the DN map $\Lambda_g$ is thus well defined for any boundary data $\psi \in H^s(\partial M)$ with $s \in \R$ under the minimal assumption (\ref{A-CondMetric1}). We can thus adress the question of the uniqueness of the metric (modulo gauge invariance) from the knowledge of the DN map. Let us prove Theorem \ref{MainThm2}, whose statement we recall here for the sake of convenience:

\begin{thm} \label{A-UniquenessCalderon}
Let $(M,g)$ and $(M,\tilde{g})$ denote two Riemannian manifolds of the form (\ref{A-Metric}) satisfying the conditions (\ref{A-CondMetric1}). Assume moreover that:
\begin{equation} \label{A-MainAssump}
 \Lambda_{g, \Gamma_j,\Gamma_j} = \Lambda_{\tilde{g}, \Gamma_j,\Gamma_j} \ \ ,\ \ {\rm{for}} \ j=0 \ {\rm{or}} \ j= 1.
\end{equation}
Then
$$
g= \tilde{g}.
$$
\end{thm}

\begin{proof}
The proof is the same as in Section \ref{3} once we have noticed that the CAM method of Section \ref{CAM} is still valid under (\ref{A-CondMetric1}). Hence we can first prove from our main assumption (\ref{A-MainAssump}) that
$$
	M(z) = \tilde{M}(z), \quad N(z) = \tilde{N}(z), \quad \forall z \in \C.
$$
We can then apply Corollary \ref{PR} and the same argument as in Section \ref{UniquenessCalderon} to conclude that $g = \tilde{g}$.
\end{proof}


\begin{rem} \label{A-Extension-MainThm}

As in Remark \ref{Extension-MainThm}, we can extend the uniqueness results of Theorem \ref{A-UniquenessCalderon} to a slightly more general situation. Precisely, let $h_0, h_1$ be measurable functions on $[0,1]$ such that
\begin{enumerate}
\item	For $j=0,1, \ h_j > 0$ a.e. in $[0,1]$,
\item $\sqrt{h_0 h_1^{n-2}}, \ \sqrt{\frac{h_0}{h_1^{n}}} \in L^1(0,1)$,
\item $\sqrt{\frac{h_0}{h_1}} \in L^1(0,1)$.
\end{enumerate}
Consider on $M = [0,1] \times K$ the class of Riemannian metrics
\begin{equation} \label{A-ExtendedMetric}
  g = h_0(x) dx^2 + h_1(x) [g_K].  		
\end{equation}

Then we have

\begin{thm} \label{A-MainThm-Extended1}
Let $g$ and $\tilde{g}$ be two metrics in the class (\ref{A-ExtendedMetric}) with $h_0, h_1$ satisfying the above assumptions. Assume that
$$
 \Lambda_{g, \Gamma_j,\Gamma_j} = \Lambda_{\tilde{g}, \Gamma_j,\Gamma_j} \ \ ,\ \ {\rm{for}} \ j=0 \ {\rm{or}} \ j= 1,
$$
for an admissible fixed energy $\lambda$. Then there exists an increasing bijection $\psi: \ [0,1] \longrightarrow [0,1]$ with $\psi$ and $\psi^{-1}$ absolutely continuous on $[0,1]$ such that
$$
  \tilde{g} = \psi^* g.
$$
\end{thm}
\end{rem}


\subsection{Even more singular metrics and invisibility.}

In this last section, we give simple counterexamples to uniqueness of the Calder\'{o}n's problem for metrics $g$ as in Theorems \ref{A-UniquenessCalderon} or \ref{A-MainThm-Extended1} but with singular coefficients which do not satisfy the assumptions of these theorems. We shall give counter-examples in the case of zero frequency that are \emph{sharp with respect to our uniqueness results}. We emphasize that these counterexamples are relatively close in spirit to the the ones given in \cite{GKLU, GKLU2, GLU}. \\

\noindent \textbf{A}. Consider first a metric $g$ of the type (\ref{A-Metric}) where
\begin{equation} \label{h1}
  \sqrt{h}(x) = h_1^{\frac{n-1}{2}}(x) = \left\{ \begin{array}{cc}
	                                             (\frac{1}{4} - x)^{-r}, & x \in [0,\frac{1}{4}[, \\
																							 f(x), & x \in [\frac{1}{4}, \frac{3}{4}], \\
																							 (x-\frac{3}{4})^{-r}, & x \in ]\frac{3}{4},1],
																								\end{array} \right.
\end{equation}
where $f(x)$ is any smooth positive function on $[\frac{1}{4}, \frac{3}{4}]$ and $r \geq 1$. Clearly, the coefficients of this metric are neither bounded near the singular surfaces $x= \frac{1}{4}$ and $x=\frac{3}{4}$, nor do they satisfy (\ref{A-CondMetric1}) since $\sqrt{h} \notin L^1(0,1)$. Such surfaces are called {\it{interfaces}}.

We look at the Dirichlet problem at zero frequency
\begin{equation} \label{DirichletPb-lambda0}
  \left\{ \begin{array}{cc} -\Delta_g u = 0, & \textrm{on} \ M, \\ u = \psi, & \textrm{on} \ \partial M. \end{array} \right.
\end{equation}
Recall that using separation of variables, we look at the solutions $u$ of (\ref{DirichletPb-lambda0}) under the form
$$
  u = \sum_{m = 0}^\infty u_{m}(x) Y_{m},
$$
and that the $u_{m}$ satisfy the ODEs with boundary conditions
\begin{equation} \label{ODE-BC-1}
  \left\{  \begin{array}{c} -\frac{1}{\sqrt{h}} \left( \sqrt{h} u_{m}' \right)' + \m u_{m} = 0, \\
	                          u_{m}(0) = \psi^0_{m}, \quad u_{m}(1) = \psi^1_{m},
					 \end{array} \right.										
\end{equation}
where $\psi = (\psi^0, \psi^1) = \ds \sum_{m=0}^\infty (\psi^0_{m}, \psi^1_{m})$ are the Fourier coefficients of the boundary data $\psi$. With our choice (\ref{h1}), this ODE becomes on the interval $(\frac{3}{4}, 1]$
\begin{equation} \label{ODE1}
  u_{m}''  - \frac{r}{(x-\frac{3}{4})} u_{m}' - \mu_m u_{m} = 0,
\end{equation}
and a similar ODE on $[0,\frac{1}{4}[$. This is is a standard Bessel equation, (see for instance \cite{Leb}, Eq. (5.4.11)).

In order to simplify the notation in the following expressions, we set $X = x-\tr$. Thanks to (\cite{Leb}, Eq. (5.4.12)), for $m \not=0$, the solutions of (\ref{ODE1}) have the form:
\begin{equation}\label{Bessel1}
u_{m} = A_{m}\  X^{\frac{1+r}{2}} \ I_{\frac{1+r}{2}} (\sqrt{\mu_m} X) + B_{m} \ X^{\frac{1+r}{2}}  K_{\frac{1+r}{2}} (\sqrt{\mu_m} X),
\end{equation}
where $A_{m}, \  B_{m}$ are some complex constants, and $I_{\nu}$, $K_{\nu}$ are the modified Bessel functions.  In the same way, for $m=0$, the solutions of (\ref{ODE1}) are:
\begin{equation}\label{Bessel10}
u_{0} = A_{0} X^{1+r} + B_{0},
\end{equation}
for some  constants $A_{0}, \ B_{0} \in \C$.

In order to determine the constants $A_{m}$ and $B_{m}$, we must impose some natural energy condition on the solution $u$ (and thus the $u_{m}$) we are looking for and of course, use the boundary condition at $x=1$. Following \cite{GKLU}, we say that $u$ is a solution of finite energy of (\ref{DirichletPb-lambda0}) if $u$ satisfies (\ref{Weak-DirichletPb}) and  $u \in H^1(M)$, i.e $\int_M (|u|^2 + |\nabla u|^2_g ) dVol_g < \infty$. This last condition is equivalent using separation of variables to
\begin{equation} \label{FiniteEnergy1}
  \sum_{m} \int_0^1 \left( |u'_{m}|^2 + \mu_m |u_{m}|^2 + h_1(x) |u_{m}|^2 \right) \sqrt{h}(x) dx < \infty.
\end{equation}

Now, we recall (see \cite{Leb}, Eqs (5.7.2)- (5.7.12)) that for $\nu>0$, we have:
\begin{equation}\label{asympt10}
I_{\nu}(z) \sim  \left( \frac{z}{2}\right)^{\nu} \ \ , \ \ K_{\nu}(z) \sim \frac{1}{2} \Gamma(\nu) \left(\frac{z}{2} \right)^{-\nu} \ ,\ z \rightarrow 0,
\end{equation}
and we can differentiate these asymptotics. Since we look for solutions $u$ satisfying (\ref{FiniteEnergy1}), we have to take $B_{m}=0$ for all $m \geq 0$. Moreover, using the boundary condition $u_{m}(1)= \psi_{m}^1$, and since all the zeros of the modified Bessel function $I_{\nu}(z), \nu>0$ are purely imaginary, we obtain immediately:
\begin{equation}\label{constants1}
 A_{m} = \frac{4^{\frac{1+r}{2}} \psi_{m}^1}{I_{\frac{1+r}{2}} ( \frac{\sqrt{\mu_m}}{4})} \quad, \quad A_{0}= 4^{1+r} \psi_{0}^1.
\end{equation}
Consequently, for $x \in ]\tr, 1]$, the solution $u$ of (\ref{DirichletPb-lambda0}) is uniquely given by (with the notation $X= x-\tr$)
\begin{equation} \label{Sol1}
u = (4X)^{1+r} \psi_{0}^1 \ Y_{0} \ + \sum_{m \not= 0} (4X)^{\frac{1+r}{2}} \psi_{m}^1 \frac{I_{\frac{1+r}{2}} (\sqrt{\mu_m} X)}{I_{\frac{1+r}{2}} ( \frac{\sqrt{\mu_m}}{4})} \ Y_{m}.
\end{equation}

Let us now show that $u$ given by (\ref{Sol1}) is a finite energy solution in $]\frac{3}{4},1] \times K$ under the assumption $\psi \in H^1(\partial M)$. We thus have to prove that (\ref{FiniteEnergy1}) is true if we integrate $x$ between $\frac{3}{4}$ and $1$. This will follow from the known estimates on the ratios of Bessel functions (\cite{JB}, Eq. (1.5)):
\begin{equation}\label{boundratio1}
\left| \frac{I_{\nu}(x)}{I_{\nu}(y)} \right|  \leq \left(\frac{x}{y} \right)^{\nu} \ \ {\rm{for}} \ 0<x<y \ {\rm{and}} \ \nu >-\half.
\end{equation}
and the following estimate on the logarithmic derivatives of the Bessel functions (\cite{JB}, Eq. (1.4)):
\begin{equation} \label{LogDerivBessel1}
\left| \frac{I_{\nu}'(x)}{I_{\nu}(x)} \right|  \leq 1 + \frac{\nu}{x} \ ,\ {\rm{for}} \ x>0\ {\rm{and}} \ \nu >0.
\end{equation}
Using (\ref{Bessel1}), (\ref{constants1}), (\ref{boundratio1}) and (\ref{LogDerivBessel1}), we obtain easily
$$
  |u_{m}| \leq C X^{1+r} |\psi^1_{m}|, \quad |u_{m}'| \leq C X^{r} |\psi^1_{m}|,
$$
from which we see that if $\psi \in H^1(\partial M)$, (and $r<3$ if $n=2$),
$$
  \sum_{m=0}^\infty \int_{\frac{3}{4}}^1 \left( |u'_{m}|^2 + \mu_m |u_{m}|^2 + h_1(x) |u_{m}|^2 \right) \sqrt{h}(x) dx < C \sum_{m=0}^\infty (1+\mu_m) |\psi^1_{m}|^2 < \infty.
$$

On the interval $[0,\frac{1}{4}[$, we have similarly a unique solution of finite energy $u$ with boundary condition at $x=0$ given by
\begin{equation} \label{Sol10}
u = (4X)^{1+r} \psi_{0}^0 \ Y_{0} \ + \sum_{m \not= 0} (4X)^{\frac{1+r}{2}} \psi_{m}^0 \frac{I_{\frac{1+r}{2}} (\sqrt{\mu_m} X)}{I_{\frac{1+r}{2}} ( \frac{\sqrt{\mu_m}}{4})} \ Y_{m},
\end{equation}
where $X= \frac{1}{4} - x$.

It remains to extend the solution $u$ on the interval $[\frac{1}{4}, \frac{3}{4}]$ where the metric is smooth and therefore, where the Sturm-Liouville equation (\ref{ODE-BC-1}) is regular. Hence all possible solutions of (\ref{DirichletPb-lambda0}) on $[\frac{1}{4}, \frac{3}{4}]$ are $H^1$ in the classical sense and thus have finite energy. Among all these solutions, we choose the one such that $u$ satisfies (\ref{Weak-DirichletPb}) \emph{globally}, that is $u$ is a weak solution of (\ref{DirichletPb-lambda0}) on $M$. Using the separability of the equation, the condition (\ref{Weak-DirichletPb}) is equivalent here to
\begin{equation} \label{BC1}
 \int_0^1 \left( \sqrt{h} u_{m}' v_{m}' + \mu_m \sqrt{h} u_{m} v_{m} \right) dx = 0, \quad \forall v_{m} \in C^\infty_0(0,1).
\end{equation}
Since $\sqrt{h} u_{m}'$ is absolutely continuous on $(0,\frac{1}{4}) \cup (\frac{1}{4},\frac{3}{4}) \cup (\frac{3}{4}, 1)$, we can integrate by parts on each of these intervals and use the equation (\ref{ODE1}) to show that (\ref{BC1}) is equivalent to
\begin{equation} \label{BC10}
  (\sqrt{h} u_{m}')(\frac{1}{4}-) = (\sqrt{h} u_{m}')(\frac{1}{4}+), \quad (\sqrt{h} u_{m}')(\frac{3}{4}-) = (\sqrt{h} u_{m}')(\frac{3}{4}+).
\end{equation}
Hence, between $\frac{1}{4}$ and $\frac{3}{4}$, we must choose the separated solution $u_{m}$ that satisfies the boundary conditions (\ref{BC10}). Using (\ref{asympt10}) and its derivative, we see immediately that
\begin{equation} \label{Neumann1}
  (\sqrt{h} u_{m}')(\frac{1}{4}-) = D_{m}^0, \quad (\sqrt{h} u_{m}')(\frac{1}{4}-) = D_{m}^1,
\end{equation}
for some constants $D_{m}^0$ and $D_{m}^1$. We thus choose $u_{m}$ between $\frac{1}{4}$ and $\frac{3}{4}$ as the unique solution of the regular ODE (\ref{ODE-BC-1}) with Neumann boundary conditions (\ref{BC10})-(\ref{Neumann1}). This solution is unique since for all $m \ne 0$, the value $-\mu_m$ is negative and thus does not belong to the Neumann spectrum of the operator $H= -\frac{1}{\sqrt{h}} \frac{d}{dx} \left( \sqrt{h} \frac{d}{dx} \right)$. For $m = 0$ however, $-\mu_0 = 0$ and thus belong to the Neumann spectrum of $H$. But the eigenvalue $0$ corresponds to the constant solutions which have been ruled out by our requirement that the solution be of finite energy, \textit{i.e.} $B_{0} = 0$.

To summarize, there is a unique solution of finite energy $u$ of (\ref{DirichletPb-lambda0}). The associated DN map can be computed exactly as in section \ref{3} and its expression on each harmonic $Y_{m}$ still takes the form
\begin{equation} \label{DN-Singular1}
  \Lambda_g^{m}(0) \left( \begin{array}{c} \psi_{m}^0  \\ \psi_{m}^1  \end{array} \right) =
  \left( \begin{array}{c} - (\sqrt{h} u_{m}')(0)  \\ \ \ (\sqrt{h} u_{m}')(1)  \end{array} \right) .
\end{equation}
Since $(\sqrt{h} u_{m}')(0)$ and $(\sqrt{h} u_{m}')(1)$ only depend on the metric on $([0,\frac{1}{4}[ \cup ]\frac{3}{4},1]) \times K$ according to (\ref{Sol1}) and (\ref{Sol10}), we see that the DN map cannot \emph{see} the metric between $\frac{1}{4}$ and $\frac{3}{4}$. In other words, any body within $([\frac{1}{4}, \frac{3}{4}]) \times K$ is cloaked from the outside world by making boundary measurements at $x=0$ and $x=1$.

\begin{rem}
We can naturally adapt the above construction to find counterexamples to uniqueness for the partial DN map $\Lambda_{g,\Gamma_1,\Gamma_1}(0)$ for instance. In this case, only one interface, say $x =\tr$ is necessary to obtain the same result, leading to the conclusion that the metric between $0$ and $\frac{3}{4}$ cannot be seen by a boundary measurement.

An illustrative example in dimension $3$ would be (take $n = 2$ and $r \in [1,3[$)
$$
  M = [0,1] \times \mathbb{S}^2,
$$
equipped with the metric ($f(x)$ is any positive smooth function)
$$
   g = \left\{ \begin{array}{cc} (x-\frac{3}{4})^{-2r} (dx^2 + g_{\mathbb{S}^2}) & x \in ]\frac{3}{4},1], \\
	                              f(x) (dx^2 + g_{\mathbb{S}^2}) & x \in [0,\frac{3}{4}]. \end{array} \right.
$$
Note that the Riemannian volume of $(M,g)$ is infinite in our examples.
\end{rem}

\vspace{0.2cm}
\noindent \textbf{B}. Consider here a metric $g$ of the type (\ref{A-Metric}) where
\begin{equation} \label{h2}
  \sqrt{h}(x) = h_1^{\frac{n-1}{2}}(x) = \left\{ \begin{array}{cc}
	                                             (\frac{1}{4} - x)^{r}, & x \in [0,\frac{1}{4}[, \\
																							 f(x), & x \in [\frac{1}{4}, \frac{3}{4}], \\
																							 (x-\frac{3}{4})^{r}, & x \in ]\frac{3}{4},1],
																								\end{array} \right.
\end{equation}
where $f(x)$ is any smooth positive function on $[\frac{1}{4}, \frac{3}{4}]$ and $r \geq 1$. Note that for this choice of metric, the condition $\frac{1}{\sqrt{h}} \in L^1$ is not satisfied. On the interval $(\frac{3}{4}, 1]$, after separation of variables, the ODE becomes the Bessel type equation
\begin{equation} \label{ODE2}
  u_{m}''  + \frac{r}{(x-\frac{3}{4})} u_{m}' - \mu_m  u_{m} = 0.
\end{equation}
If we set as before $X = x-\tr$, the solutions of (\ref{ODE2}) have the expression:
\begin{equation}\label{Bessel2}
u_{m} = A_{m}\  X^{\frac{1-r}{2}} \ I_{-\frac{1-r}{2}} (\sqrt{\mu_m  } X) + B_{m} \ X^{\frac{1-r}{2}}  K_{-\frac{1-r}{2}} (\sqrt{\mu_m } X),
\end{equation}
where $A_{m}, \  B_{m}$ are some complex constants and for $m=0$, the solutions of (\ref{ODE2}) are:
\begin{equation}\label{Bessel20}
u_{0} = A_{0} + B_{0} X^{1-r}, \ ({\rm{if}} \ r>1), \ \ \ u_{0} = A_{0} + B_{0} \log X ,  \ ({\rm{if}} \ r=1),
\end{equation}
for some  constants $A_{0}, \ B_{0} \in \C$.

Requiring that $u$ be a solution of finite energy, we get immediately that $B_{m} = 0$ for all $m \geq 0$. Using the boundary condition at $x=1$, we conclude as in the previous example that there is a unique solution of (\ref{DirichletPb-lambda0}) on $(\frac{3}{4}, 1]$. The situation is exactly the same between $0$ and $\frac{1}{4}$, \textit{i.e.} using the boundary condition at $x=0$, we can show that there is a unique solution of (\ref{DirichletPb-lambda0}) on $[0,\frac{1}{4}[$. Micmicking the last step of the previous proof, we conclude that there is a unique \emph{global} weak solution $u$ of (\ref{DirichletPb-lambda0}) of finite energy. Note that the solution between $\frac{1}{4}$ and $\frac{3}{4}$ must be identically $0$ if $r>1$. Finally, the same calculations as in the previous example show that the DN map only depends on the metric outside the interfaces $x=\frac{1}{4}$ and $x=\frac{3}{4}$ achieving thus a cloaking region between $\frac{1}{4}$ and $\frac{3}{4}$.

An example in dimension $3$ would be (take $n = 2$ and $r \geq 1$)
$$
  M = [0,1] \times \mathbb{S}^2,
$$
equipped with the metric ($f(x)$ is any positive smooth function)
$$
   g = \left\{ \begin{array}{cc} (x-\frac{3}{4})^{2r} (dx^2 + g_{\mathbb{S}^2}) & x \in ]\frac{3}{4},1], \\
	                              f(x) (dx^2 + g_{\mathbb{S}^2}) & x \in [\frac{1}{4},\frac{3}{4}], \\
																(\frac{1}{4}-x)^{2r} (dx^2 + g_{\mathbb{S}^2}) & x \in [0,\frac{1}{4},1[
																 \end{array} \right.
$$
Note at last that the Riemannian volume of $(M,g)$ is finite in this class of examples. \\

\vspace{0.2cm}
\noindent \textbf{C}. Consider here a metric $g$ of the type (\ref{A-ExtendedMetric}) where
\begin{equation} \label{h3}
  h_0(x) = 1, \quad h_1(x) = \left\{ \begin{array}{cc}
	                                             (\frac{1}{4} - x)^{r}, & x \in [0,\frac{1}{4}[, \\
																							 f(x), & x \in [\frac{1}{4}, \frac{3}{4}], \\
																							 (x-\frac{3}{4})^{r}, & x \in ]\frac{3}{4},1],
																								\end{array} \right.
\end{equation}
where $f(x)$ is any smooth positive function on $[\frac{1}{4}, \frac{3}{4}]$ and $\frac{2}{n} \leq r < 2$.  Note that for this choice of metric the condition $\sqrt{\frac{h_0}{h_1^{n}}} \in L^1(0,1)$ is not satisfied (see Remark \ref{A-Extension-MainThm}). A short calculation shows that the Dirichlet problem (\ref{DirichletPb-lambda0}) associated to such metric amounts to (after separation of variables) the following radial ODEs
\begin{equation} \label{ODE-C}
  u_{m}'' + \frac{h'(x)}{2h(x)} u_{m}' - \mu_m  u_{m} = 0, \quad h(x) = \frac{h_1^{n}}{h_0}.
\end{equation}
Moreover, we recall that we look for solutions $u$ of finite energy which in this setting is equivalent to
\begin{equation} \label{FE-C}
  \sum_{m=0}^\infty \int_0^1 \left[ |u'_{m}|^2 + |u_{m}|^2 + \mu_m  \frac{1}{h_1(x)} |u_{m}|^2 \right] h_1^{\frac{n}{2}} dx < \infty.
\end{equation}

On the interval $(\frac{3}{4}, 1]$, the radial ODE (\ref{ODE3}) becomes the Bessel type equation
\begin{equation} \label{ODE3}
  u_{m}''  + \frac{nr}{2(x-\frac{3}{4})} u_{m}' - \frac{\mu_m }{X^r} u_{m} = 0.
\end{equation}
If we set as before $X = x-\tr$, the solutions of (\ref{ODE3}) have the form (\cite{Leb}, Eq. (5.4.12))
\begin{equation}\label{Bessel3}
u_{m} = A_{m}\  X^{\frac{2-nr}{4}} \ I_{\frac{nr-2}{2(2-r)}} (\frac{2\sqrt{\mu_m }}{2-r} X^{\frac{2-r}{2}}) + B_{m} \ X^{\frac{2-nr}{4}} \ K_{\frac{nr-2}{2(2-r)}} (\frac{2\sqrt{\mu_m }}{2-r} X^{\frac{2-r}{2}}),
\end{equation}
where $A_{m}, \  B_{m}$ are some complex constants and for $m=0$, the solutions of (\ref{ODE2}) are:
\begin{equation}\label{Bessel30}
u_{0} = A_{0} + B_{0} X^{\frac{2-nr}{2}}, \ ({\rm{if}} \ nr>2), \ \ u_{0} = A_{0} + B_{0} \log X, \ ({\rm{if}} \ nr=2).
\end{equation}
for some  constants $A_{0}, \ B_{0} \in \C$.

Requiring that $u$ satisfies (\ref{FE-C}), we get immediately that $B_{m} = 0$ for all $m \geq 0$. Using the boundary condition at $x=1$, we conclude as in example A that there is a unique solution of (\ref{DirichletPb-lambda0}) on $(\frac{3}{4}, 1]$. The situation is similar between $0$ and $\frac{1}{4}$. Micmicking the last step of the previous proof, we conclude that there is a unique global weak solution $u$ of (\ref{DirichletPb-lambda0}) with finite energy. Note that the solution between $\frac{1}{4}$ and $\frac{3}{4}$ must be identically $0$ here. At last, the same calculations as in the previous example show that the DN map only depends on the metric outside the interfaces $x=\frac{1}{4}$ and $x=\frac{3}{4}$ achieving thus a cloaking region between $\frac{1}{4}$ and $\frac{3}{4}$. Note at last that the Riemannian volume of $(M,g)$ is finite in this class of examples.

An example in dimension $3$ would be (take $n = 2$ and $1 \leq r < 2$)
$$
  M = [0,1] \times \mathbb{S}^2,
$$
equipped with the metric ($f(x)$ is any positive smooth function)
$$
   g = \left\{ \begin{array}{cc} dx^2 + (x-\frac{3}{4})^{r} \ g_{\mathbb{S}^2} & x \in ]\frac{3}{4},1], \\
	                               dx^2 + f(x) \ g_{\mathbb{S}^2}  & x \in [\frac{1}{4},\frac{3}{4}], \\
																 dx^2 + (\frac{1}{4}-x)^{r} \ g_{\mathbb{S}^2} & x \in [0,\frac{1}{4},1[
																 \end{array} \right.
$$

\vspace{0.2cm}
\noindent \textbf{D}. Consider here a metric $g$ of the type (\ref{A-ExtendedMetric}) where
\begin{equation} \label{h4}
  h_0(x) = 1, \quad h_1(x) = \left\{ \begin{array}{cc}
	                                             (\frac{1}{4} - x)^{2}, & x \in [0,\frac{1}{4}[, \\
																							 f(x), & x \in [\frac{1}{4}, \frac{3}{4}], \\
																							 (x-\frac{3}{4})^{2}, & x \in ]\frac{3}{4},1],
																								\end{array} \right.
\end{equation}
where $f(x)$ is any smooth positive function on $[\frac{1}{4}, \frac{3}{4}]$. Note that for this choice of metric the condition $\sqrt{\frac{h_0}{h_1}} \in L^1(0,1)$ is not satisfied (see Remark \ref{A-Extension-MainThm}). Using exactly the same strategy as in the previous example, we can show that there is a unique solution of finite energy of (\ref{DirichletPb-lambda0}) on $M$. Moreover, the DN map does not distinguish the metric between $x=\frac{1}{4}$ and $x=\frac{3}{4}$.

An example in dimension $3$ would be (take $n = 2$)
$$
  M = [0,1] \times \mathbb{S}^2,
$$
equipped with the metric ($f(x)$ is any positive smooth function)
$$
   g = \left\{ \begin{array}{cc} dx^2 + (x-\frac{3}{4})^{2} g_{\mathbb{S}^2} & x \in ]\frac{3}{4},1], \\
	                               dx^2 + f(x) g_{\mathbb{S}^2} & x \in [\frac{1}{4},\frac{3}{4}], \\
																 dx^2 + (\frac{1}{4}-x)^{2} g_{\mathbb{S}^2} & x \in [0,\frac{1}{4},1[
																 \end{array} \right.
$$

\noindent \footnotesize{DEPARTEMENT DE MATHEMATIQUES. UMR CNRS 8088. UNIVERSITE DE CERGY-PONTOISE. 95302 CERGY-PONTOISE. FRANCE. \\
\emph{Email adress}: thierry.daude@u-cergy.fr \\

\noindent DEPARTMENT OF MATHEMATICS AND STATISTICS. MCGill UNIVERSITY. MONTREAL, QC, H3A 2K6, CANADA \\
\emph{Email adress}: nkamran@math.mcgill.ca \\

\noindent LABORATOIRE DE MATHEMATIQUES JEAN LERAY. UMR CNRS 6629. 2 RUE DE LA HOUSSINIERE BP 92208. F-44322 NANTES CEDEX 03 \\
\emph{Email adress}: francois.nicoleau@math.univ-nantes.fr}

\end{document}